\documentclass[12pt,reqno]{amsart}
\usepackage[width=125mm, height=240mm, top=30mm, left=35mm, right=22mm, headsep=6mm, a4paper]{geometry}
\usepackage{mathrsfs,amssymb,amsmath,amsthm}%

\usepackage{hyperref}
\hypersetup{pdftex,
            colorlinks,
            linkcolor=blue,
            citecolor=blue,
            hyperindex,
            pdfstartview=FitH,
            plainpages=false,
            bookmarks,
            bookmarksopen,
            bookmarksnumbered,
            pdfauthor={Chengchun Hao, Hailiang Li},
            pdftitle={well-posedness for a viscous liquid-gas two-phase flow model},
            pdfkeywords={},
            pdfsubject={}
            }
\allowdisplaybreaks
\numberwithin{equation}{section}
\newtheorem{theorem}{Theorem}[section]
\newtheorem{corollary}[theorem]{Corollary}
\newtheorem{lemma}[theorem]{Lemma}
\newtheorem{proposition}[theorem]{Proposition}
\theoremstyle{definition}
\newtheorem{definition}{Definition}[section]
\theoremstyle{remark}
\newtheorem{remark}[theorem]{Remark}

\newcommand{\R}{\Bbb{R}}
\newcommand{\N}{\Bbb{N}}
\newcommand{\Z}{\mathbb Z}
\newcommand{\F}{\mathscr{F}}
\newcommand{\Sz}{\mathscr{S}}
\newcommand{\be}[1][s]{B^{#1}}
\newcommand{\hbe}[2][s-1]{B^{#1,#2}}
\newcommand{\LL}[1][\infty]{\tilde{L}^#1}

\newcommand{\set}[1]{\{#1\}}
\newcommand{\abs}[1]{|#1|}

\newcommand{\norm}[1]{\|#1\|}
\newcommand{\lnorm}[1]{\left\Vert#1\right\Vert}

\newcommand{\ls}{\leqslant}
\newcommand{\gs}{\geqslant}
\newcommand{\mb}{\bar{m}}
\newcommand{\nb}{\bar{n}}
\newcommand{\uu}{\mathbf{u}}
\newcommand{\vv}{\mathbf{v}}
\newcommand{\ww}{\mathbf{w}}
\newcommand{\GG}{\mathbf{G}}
\newcommand{\HH}{\mathbf{H}}
\newcommand{\hn}{{d/2}}
\newcommand{\J}{\mathbb{F}_\ell}
\newcommand{\eps}{\varepsilon}
\newcommand{\dive}{\mathrm{div}}

\newcommand{\dk}[1][]{\triangle_k #1}

\newcommand{\doi}[1]{\href{http://dx.doi.org/#1}{doi:#1}}

\newcommand{\sameauthor}{------}

\title[Well-posedness of a viscous liquid-gas flow model]{Well-posedness for a multi-dimensional viscous liquid-gas two-phase flow model}

\author[C. C. Hao]{Chengchun Hao}
\address{Institute of Mathematics,
   Academy of Mathematics \& Systems Science\\
   and Hua Loo-Keng Key Laboratory of Mathematics, Chinese Academy of Sciences\\
 Beijing 100190, China}
\email{hcc@amss.ac.cn}

\author[H.-L. Li]{Hai-Liang Li}
\address{Department of Mathematics,
Capital Normal University\\
 Beijing 100048, China}
\email{hailiang.li.math@gmail.com}

\begin{document}

\begin{abstract}
The Cauchy problem of a multi-dimensional ($d\gs 2$) compressible
viscous liquid-gas two-phase flow model is concerned in this paper.
We investigate the global existence and uniqueness of the strong
solution for the initial data close to a stable equilibrium and the
local in time existence and uniqueness of the solution with general
initial data in the framework of Besov spaces. A continuation
criterion is also obtained for the local solution.
\end{abstract}



\maketitle

\section{Introduction}

The models of two-phase or multi-phase flows have a very broad
applications of hydrodynamics in industry, engineering, biomedicine
and so on, where the fluids under investigation contain more than
one component.
Indeed, it has been estimated that over half of anything which is
produced in a modern industrial society depends, to some extent, on
a multi-phase flow process for their optimum design and safe
operations. In nature, there is a variety of different multi-phase
flow phenomena, such as sediment transport, geysers, volcanic
eruptions, clouds and rain. In addition, the models of multi-phase
flows also naturally appear in many contexts within biology, ranging
from tumor biology and anticancer therapies, development biology and
plant physiology, etc. The principles of single-phase flow fluid
dynamics and heat transfer are relatively well understood, however,
the thermo-fluid dynamics of two-phase flows is an order of
magnitude more complicated subject than that of the single-phase
flow due to the existence of moving and deformable interface and its
interactions with two phases~\cite{Brennen,IshiiHibiki,Kolev1}.

We consider the drift-flux model of two-phase flows in the present
paper, which is principally developed by Zuber and
Findlay (1965), Wallis (1969) and Ishii (1977). The basic idea about
drift-flux models is that both phases are well mixed, but the
relative motion between the phases is governed by a particular
subset of the flow parameters. In the case of liquid-gas fluids, it
relates the liquid-gas velocity difference to the drift-flux (or
``drift velocity'') of the vapor relative to the liquid due to
buoyancy effects. In general, the drift-flux models consist of two
mass conservation equations corresponding to each of the two phases,
and one equation for the conservation of the mixture momentum, and
are particularly useful in the analysis of sedimentation,
fluidization (batch, cocurrent and countercurrent), and so on
(\cite{Ishii,Wallis,Zuber,ZuberFindlay}).

The Cauchy problem to a simplified version of the viscous compressible liquid-gas two-phase
flow model of drift-flux type in $\R^d$ ($d\gs 2$), where the gas phase has not been taken into account in the momentum equation except the pressure term and the equal velocity of the liquid and gas flows has been assumed, reads
\begin{align}\label{eq.system}
  \left\{\begin{aligned}
    &\tilde{m}_t+\dive(\tilde{m}\uu)=0,\\
    &\tilde{n}_t+\dive(\tilde{n}\uu)=0,\\
    &(\tilde{m}\uu)_t+\dive(\tilde{m}\uu\otimes \uu)
     +\nabla P(\tilde{m},\tilde{n})
     =\tilde{\mu}\Delta \uu+(\tilde{\mu}+\tilde{\lambda})\nabla \dive \uu,
  \end{aligned}
  \right.
\end{align}
with the initial data
\begin{align}\label{eq.data}
  (\tilde{m},\tilde{n},\uu)|_{t=0}=(\tilde{m}_0,\tilde{n}_0,\uu_0)(x), \quad \text{in }
  \R^d,
\end{align}
%
where $\tilde{m}=\alpha_l\rho_l$ and $\tilde{n}=\alpha_g\rho_g$
denote the liquid mass and the gas mass, respectively. The unknowns
$\alpha_l$, $\alpha_g\in [0,1]$ denote the liquid and gas volume
fractions, satisfying the fundamental relation:
$\alpha_l+\alpha_g=1$. The unknown variables $\rho_l$ and $\rho_g$
denote the liquid and gas densities, satisfying the equations of
states $\rho_l=\rho_{l,0}+(P-P_{l,0})/a_l^2$, $\rho_g=P/a_g^2$,
where $a_l$ and $a_g$ denote the sonic speeds of the liquid and the
gas, respectively, and $P_{l,0}$ and $\rho_{l,0}$ are the reference
pressure and density given as constants.
$\uu$ denotes the mixed velocity of the liquid and the gas, and $P$
is the common pressure for both phases, which satisfies
\begin{align}\label{eq.pressure}
  P(\tilde{m},\tilde{n})=C_0\left(-b(\tilde{m},\tilde{n})+\sqrt{b^2(\tilde{m},\tilde{n}) +c(\tilde{m},\tilde{n})}\right),
\end{align}
with $C_0=a_l^2/2$, $k_0=\rho_{l,0}-P_{l,0}/a_l^2>0$, $a_0=a_g^2/a_l^2$ and
\begin{align*}
  b(\tilde{m},\tilde{n})=k_0-\tilde{m}-a_0\tilde{n}, \quad c(\tilde{m},\tilde{n})=4k_0a_0\tilde{n}.
\end{align*}
$\tilde{\mu}$ and $\tilde{\lambda}$ are the viscosity constants,
satisfying
\begin{align}\label{eq.mu.1}
  \tilde{\mu}>0,\quad 2\tilde{\mu}+d\tilde{\lambda}\gs 0.
\end{align}

For the one-dimensional case, the existence and/or uniqueness of the
global weak solution to the free boundary value problem was studied
in \cite{EFF09,EK09,YZ09,YZ10} where the liquid is incompressible
and the gas is polytropic, and in \cite{EK08} where both of two
fluids are compressible.
However, there are few results for multi-dimensional cases except for some computational results \cite{PT07}. As a
generalization of the results in \cite{EK08}, the existence of the
global solution to the 2D model was obtained in \cite{YZZ10} for
small initial energies. In \cite{YZZ11}, a blow-up criterion for the
2D model was proved in terms of the upper bound of the liquid mass
for the strong solution in a smooth bounded domain.

One of the main results of the present paper is the
existence and uniqueness of the  global strong solution to the
Cauchy problem~\eqref{eq.system}--\eqref{eq.data} under the
framework of Besov spaces, for all multi-dimensions $d\gs 2$,
provided that the initial data are close to a constant equilibrium
state. The other result is the local
well-posedness and the continuation criterion to the Cauchy problem
with general initial data. Because of the similarity of the viscous
liquid-gas two-phase flow model to the compressible Navier-Stokes
equations, we can apply some ideas adopted in the proof of
well-posedness for the compressible Navier-Stokes equations to deal
with the two-phase flow model. It is Danchin who first makes
important progress in applying the Littlewood-Paley theory and Besov
spaces to sovle the existence and uniqueness for the compressible
Navier-Stokes equations or barotropic viscous fluid in
\cite{Danchin00, Danchin07} and for the flows of compressible
viscous and heat-conductive gases in \cite{Danchin01,Danchin01a}.
However, it is non-trivial to apply directly the ideas used in
single-phase models into the two-phase models because the momentum
equation is given only for the mixture and that the pressure
involves the masses of two phases in a nonlinear way, which makes it
rather difficult to obtain the estimates of the masses and the mixed
velocity $(\tilde{m},\tilde{n},\uu)$ in Lebesgue spaces $L^p$ with
respect to the time. In addition, it seems impossible to get the
estimates of $\tilde{m}$ and $\tilde{n}$ from the system
simultaneously due to  the strong coupling among the
corresponding terms, even if we change the variables
$(\tilde{m},\tilde{n})$ linearly.

To overcome these difficulties in global well-posedness theory, we make use of a nonlinear variable transform so as to separate the two mass variables from each other,
which enable us to decompose the original system into a transport
equation and a coupled hyperbolic-parabolic system. To be more
precisely, we first divide the momentum equation by $\tilde{m}$
(which supplies additional information) and take a new variable
$n=a_0(\tilde{n}/\tilde{m}-\nb/\mb)$ for some constants $\nb$ and $\mb$. This makes the resulted equation for $n$
a homogeneous transport equation with the velocity $\uu$, and the
expected estimates of the new variable depend only on the mixed
velocity. Then, we remove the linear terms involving $n$ from the
momentum equation so as to separate linearly the equation about $n$
from the others, which can be done by virtue of the variables
changes with a careful choice of coefficient factors. Finally, to
establish the \emph{a priori} estimates for the global existence theory,
we deal with the linearized system directly instead of separating
the velocity into the compressible and incompressible parts.

As for the local well-posedness theory for general data, we need to
reformulate the original system and deal with the resulted
nonlinear system directly, and in terms of the improvement of the
\emph{a priori} estimates on the densities, we can generalize the local
well-posedness result in \cite{BCDbook,Danchin07} to the two-phase flow
model \eqref{eq.system} with the  specified pressure function.

Before stating the main results, we introduce some notations.
Throughout the paper, $C$ (or $c$) stands for a harmless constant,
and we sometimes use $A\lesssim B$ to stand for $A\ls CB$. $B^s$ and
$B^{s,t}$ denote usual homogeneous Besov spaces and hybrid Besov
spaces, respectively; $\LL(B^{s,t})$ and $\tilde{\mathcal{C}}(B^{s,t})$ are mixed time-spatial spaces, see the appendix for details. Let us now
introduce the functional spaces which appear in the theorems.
\begin{definition}
 For $T>0$ and $s\in\R$, we denote
\begin{align*}
  E_T^s=\big\{(m,n,\uu):\; &n\in \mathcal{C}([0,T];B^{s-1,s}(\R^d))\\
  &m\in\mathcal{C}([0,T];B^{s-1,s}(\R^d))\cap L^1([0,T];B^{s+1,s}(\R^d))\\
  &\uu\in\big(\mathcal{C}([0,T];B^{s-1}(\R^d))\cap L^1([0,T];B^{s+1}(\R^d))\big)^d\big\},
\end{align*}
and
\begin{align*}
  \norm{(m,n,\uu)}_{E_T^s}=\norm{n}_{\LL([0,T];B^{s-1,s})} &+\norm{m}_{\LL([0,T];B^{s-1,s})} +\norm{\uu}_{\LL([0,T];B^{s-1})} \\ &+\norm{m}_{L^1([0,T];B^{s+1,s})}+\norm{\uu}_{L^1([0,T];B^{s+1})}.
\end{align*}

We use the notation $E^s$ if $T=+\infty$, changing $[0,T]$ into $[0,\infty)$ in the definition above.
\end{definition}

\begin{definition}
  Let $\alpha\in[0,1]$ and $T>0$, denote
\begin{align*}
  F_T^\alpha:= &(\tilde{\mathcal{C}}([0,T];B^{d/2,d/2+\alpha}))^{1+1}\\
  &\qquad\times(\tilde{\mathcal{C}}([0,T]; B^{d/2-1,d/2-1+\alpha})\cap L^1([0,T];B^{d/2+1,d/2+1+\alpha}))^d.
\end{align*}
\end{definition}

Now, we state the global well-posedness results briefly as follows.
For more information about the solution, one can see Theorem
\ref{thm.2} in the second section.
\begin{theorem}[Global well-posedness for small data]\label{thm.1}
Let $d\gs 2$, $\nb\gs 0$, $\mb>(1-\mathrm{sgn}\nb)k_0$, $\tilde{\mu}>0$ and
$2\tilde{\mu}+d\tilde{\lambda}\gs 0$, in addition,
$\tilde{\mu}+\tilde{\lambda}>0$ if $d=2$. There exist two positive
constants $\sigma$ and $Q$ such that if $\tilde{m}_0-\mb$, $\tilde{n}_0-\nb\in
B^{d/2-1,d/2}$ and
$\uu_0\in B^{d/2-1}$ satisfying
  \begin{align}\label{eq.thm.1}
\norm{\tilde{m}_0-\mb}_{B^{d/2-1,d/2}} +\norm{\tilde{n}_0-\nb}_{B^{d/2-1,d/2}}+\norm{\uu_0}_{B^{d/2-1}}\ls\sigma,
  \end{align}
  then the following results hold

  {\rm (i)} Existence: The system \eqref{eq.system} has a solution $(\tilde{m},\tilde{n},\uu)$  satisfying
  $$\tilde{m}-\mb,\; \tilde{n}-\nb\in \mathcal{C}\left(\R^+;B^{d/2-1,d/2}\right), \quad \uu\in \mathcal{C}\left(\R^+;B^{d/2-1}\right),$$
   and moreover,
  \begin{align*}
  &\norm{(a(\tilde{m}-\mb)+ba_0 (\tilde{n}/\tilde{m}-\nb/\mb),\tilde{n}/\tilde{m}-\nb/\mb,\uu)}_{E^{d/2}}\\
  \ls& Q\big(\norm{\tilde{m}_0-\mb}_{B^{d/2-1,d/2}}+\norm{\tilde{n}_0-\nb}_{B^{d/2-1,d/2}} +\norm{\uu_0}_{B^{d/2-1}}\big),
  \end{align*}
  where the constants $a$ and $b$ are defined by
  \begin{align}\label{eq.ab}
  \begin{aligned}
    a= & \frac{1}{\mb^2}\left(a_0\nb+\mb+\frac{(\mb-a_0\nb)(\mb-a_0\nb-k_0)}{ \sqrt{(\mb+a_0\nb-k_0)^2+4k_0a_0\nb}}\right)>0, \\  b= & 1+\frac{(\mb+a_0\nb+k_0)}{\sqrt{(\mb+a_0\nb-k_0)^2+4k_0a_0\nb}}>0.
  \end{aligned}
\end{align}

  {\rm (ii)} Uniqueness: Uniqueness holds in $\mathcal{C}\left(\R^+; (B^{d/2-1,d/2})^{1+1}\times (B^{d/2})^d\right)$ if $d\gs 3$. If $d=2$, one should also suppose that $\tilde{m}_0-\mb$, $\tilde{n}_0-\nb\in B^{\eps,1+\eps}$ and $\uu_0\in B^{\eps}$ for a $\eps\in(0,1)$, to get uniqueness in $\mathcal{C}(\R^+; (B^{0,1})^{1+1}\times (B^1)^d)$.
\end{theorem}

For the general data bounded away from the infinity and the vacuum,
we have the following local well-posedness theory (one can refer to
Theorem \ref{thm.4} for the corresponding statement in terms of new
variables).

\begin{theorem}[Local well-posedness for general data]\label{thm.3}
  Let $d\gs 2$, $\tilde{\mu}>0$, $2\tilde{\mu}+d\tilde{\lambda}\gs 0$, the constants $\mb>0$ and $\nb\gs 0$. Assume that $\tilde{m}_0^{-1}-\mb^{-1}\in B^{d/2,d/2+1}$, $\tilde{n}_0-\nb\in B^{d/2,d/2+1}$ and $\uu_0\in B^{d/2-1,d/2}$. In addition, $\sup_{x\in\R^d} \tilde{m}_0(x)<\infty$. Then there exists a positive time $T$ such that the system \eqref{eq.system} has a unique solution $(\tilde{m},\tilde{n},\uu)$ on $[0,T]\times\R^d$ and that $(\tilde{m}^{-1}-\mb^{-1},\tilde{n}-\nb,\uu)$ belongs to $F_T^1$ and satisfies $\sup_{(t,x)\in[0,T]\times\R^d} \tilde{m}(t,x)<\infty$.
\end{theorem}

We also have the following continuation criterion for the local existence of the solution (see also Proposition \ref{prop.cc}).

\begin{theorem}[Continuation criterion]
    Under the hypotheses of Theorem \ref{thm.3}, assume that the system \eqref{eq.system} has a solution $(\tilde{m},\tilde{n},\uu)$ on $[0,T)\times\R^d$ such that $(\tilde{m}^{-1}-\mb^{-1},\tilde{n}-\nb,\uu)$ belongs to $F_{T'}^1$ for all $T'<T$ and satisfies
  \begin{align*}
   & \tilde{m}^{-1}-\mb^{-1},\tilde{n}-\nb \in L^\infty([0,T);B^{d/2,d/2+1}),\\
    &  \sup_{(t,x)\in [0,T)\times\R^d} \tilde{m}(t,x)<\infty, \quad \int_0^T \norm{\nabla\uu}_{\infty}dt<\infty.
  \end{align*}
  Then, there exists some $T^*>T$ such that $(\tilde{m},\tilde{n},\uu)$ may be continued on $[0,T^*]\times\R^d$ to a solution of \eqref{eq.system} such that $(\tilde{m}^{-1}-\mb^{-1},\tilde{n}-\nb,\uu)$ belongs to $F_{T^*}^1$.
\end{theorem}

\begin{remark}
  The results of the present paper are independent of the special structure \eqref{eq.pressure} of the nonlinear
pressure term $P$. Indeed, the similar results hold true as long as the term $\nabla P/\tilde{m}$ can be decomposed into a linear term involving the modified mass and some  nonlinear terms, similarly as \eqref{eq.2.1} in the next section.
\end{remark}

The rest of this paper is organized as follows. In Section
\ref{sec.gwp}, we investigate the global well-posedness of the
Cauchy problem. We first reformulate the system through changing
variables in order to obtain \emph{a priori} estimates in the
subsection \ref{subsec.gwp1}. In the subsection \ref{subsec.gwp2},
we are devoted to deriving \emph{a priori} estimates for the
transport equation and the linear coupled parabolic-hyperbolic
system with convection terms. The subsection \ref{subsec.gwp3}
involves the proof of the existence and uniqueness of the solution.
In Section \ref{sec.lwp}, we prove the local well-posedness of the
problem through some subsections similar to the global ones. An
appendix is devoted to recalling some properties of the
Littlewood-Paley decomposition and Besov spaces which we have used
in this paper.

\section{Global well-posedness for small data}\label{sec.gwp}

\subsection{Reformulation of the system}\label{subsec.gwp1}

Let $\nb\gs 0$ and $\mb>(1-\mathrm{sgn}\nb)k_0$, we introduce new variables
$n=a_0(\tilde{n}/\tilde{m}-\nb/\mb)$ and $m=a(\tilde{m}-\mb)+b n$,
i.e. $\tilde{m}=\mb+(m-b n)/a$ in order to cancel the
linear terms involving one modified mass from the momentum equation, where
$a$ and $b$ are positive constants defined in \eqref{eq.ab}. We also denote $n_0=a_0(\tilde{n}_0/\tilde{m}_0-\nb/\mb)$ and
$m_0=a(\tilde{m}_0-\mb)+bn_0$ throughout the sections for the global well-posedness theory.
Then, we have
\begin{align*}
  \frac{P}{C_0}=\left(1+\frac{a_0\nb}{\mb}+n\right)\tilde{m}-k_0
  +\sqrt{\left(\left(1+\frac{a_0\nb}{\mb}+n\right) \tilde{m}-k_0\right)^2 +4k_0\left(n+\frac{a_0\nb}{\mb}\right)\tilde{m}}.
\end{align*}
Taking the gradient of both sides, we get
\begin{align}\label{eq.2.1}
  \frac{\nabla P}{C_0\tilde{m}}=&\nabla m+\HH(m,n),
\end{align}
where the nonlinear term is
\begin{align*}
  \HH(m,n):=&\frac{\nabla m-b\nabla n}{a^2\mb^2\tilde{m}}\Big(-(a_0\nb+\mb)m+(a\mb^2+b(a_0\nb+\mb))n\Big)\\
  &+(K(m,n)-K(0,0))\Big\{(n+\frac{a_0\nb}{\mb}+1)\tilde{m}\nabla n+(n+\frac{a_0\nb}{\mb}+1)^2\nabla \tilde{m}\\
  &\qquad\qquad\qquad\qquad\qquad+k_0(n+\frac{a_0\nb}{\mb}-1)\frac{\nabla m-b\nabla n}{a\tilde{m}}+k_0\nabla n\Big\}\\
  &+K(0,0)\left\{(n+\frac{a_0\nb}{\mb}+1)\frac{\nabla n(m-bn)}{a}+\mb n\nabla n\right.\\
  &\qquad\qquad\qquad+[n^2+2n(\frac{a_0\nb}{\mb}+1)]\frac{\nabla m-b\nabla n}{a}+\frac{k_0}{a\tilde{m}}n(\nabla m-b\nabla n)\\
  &\qquad\qquad\qquad\left.-k_0(\frac{a_0\nb}{\mb}-1)\frac{(m-bn)(\nabla m-b\nabla n)}{a^2\mb\tilde{m}}\right\}.
\end{align*}
Here,
\begin{align*}
  K(m,n)=\frac{1}{\sqrt{\left[\left(\mb+\frac{m-bn}{a}\right)\left(n+\frac{a_0\nb}{\mb}+1\right) -k_0\right]^2 +4 k_0\left(n+\frac{a_0\nb}{\mb}\right)\left(\mb+\frac{m-bn}{a}\right)}},
\end{align*}
and $K(0,0)=1/\sqrt{(\mb+a_0\nb-k_0)^2+4k_0a_0\nb}>0$.

Therefore, with the new unknowns, we can rewrite the Cauchy problem
of the system \eqref{eq.system} as follows
\begin{align}\label{eq.reform.system}
  \left\{\begin{aligned}
    &n_t+\uu\cdot\nabla n=0,\\
    &m_t+\uu\cdot\nabla m+a\mb\dive\uu=F(m,n,\uu),\\
    &\uu_t+\uu\cdot\nabla\uu-\mu\Delta\uu-(\mu+\lambda)\nabla\dive\uu+C_0\nabla m=\GG(m,n,\uu),\\
    &(m,n,\uu)|_{t=0}=(m_0,n_0,\uu_0),
  \end{aligned}\right.
\end{align}
where
\begin{gather*}
\mu= \tilde{\mu}/\mb,\quad \lambda=\tilde{\lambda}/\mb, \quad F(m,n,\uu)=-(m-bn)\dive\uu,\\
 \GG(m,n,\uu)= -C_0\HH(m,n) -\frac{m-bn}{a\tilde{m}}(\mu\Delta\uu+(\mu+\lambda) \nabla\dive\uu).
\end{gather*}
Note here that the first equation in \eqref{eq.reform.system} is a
homogeneous transport equation, the estimates of $n$ depend only on
those of the velocity $\uu$. The second and the third ones  in
\eqref{eq.reform.system} consist of  a coupled parabolic-hyperbolic
system with the modified mass $m$ and mixed velocity $\uu$ involved.
Thus, with the help of the decomposition \eqref{eq.2.1}, the original system is decoupled into a transport equation for the modified gas flow and
a coupled system for the motion of the modified liquid fluid.

We can get the following result for the reformulated system.
\begin{theorem}\label{thm.2}
  Let $d\gs 2$,  $\nb\gs 0$, $\mb>(1-\mathrm{sgn}\nb)k_0$, $\mu>0$ and $2\mu+d\lambda\gs 0$, in addition, $\mu+\lambda>0$ if $d=2$. There exist two positive constants $\eta$ and $Q$ such that if $m_0$, $n_0\in B^{d/2-1,d/2}$ and $\uu_0\in B^{d/2-1}$ satisfying
\begin{equation}\label{thm.2.a}
\norm{m_0}_{B^{d/2-1,d/2}}+\norm{n_0}_{B^{d/2-1,d/2}}+\norm{\uu_0}_{B^{d/2-1}}\ls\eta,
\end{equation}
  then the following results hold:

  {\rm (i)} Existence: The system \eqref{eq.reform.system} has a solution $(m,n,\uu)$ in $E^{d/2}$ which satisfies
  $$\norm{(m,n,\uu)}_{E^{d/2}}\ls Q\big(\norm{m_0}_{B^{d/2-1,d/2}}+\norm{n_0}_{B^{d/2-1,d/2}}+\norm{\uu_0}_{B^{d/2-1}}\big).$$
  It also belongs to the affine space $$(m_L,n_0,\uu_L)+(\mathcal{C}^{1/2}(\R^+;B^{d/2-1}))^{1+1}\times (\mathcal{C}^{1/8}(\R^+; B^{d/2-5/4}))^d,$$
   where $(m_L,\uu_L)$  is the solution of the linear system
  \begin{align}\label{eq.system.app3}
  \left\{\begin{aligned}
    &\partial_t m_L+a\mb\dive \uu_L=0,\\
    &\partial_t\uu_L-\mu\Delta\uu_L -(\mu+\lambda)\nabla\dive\uu_L +C_0\nabla  m_L=0,\\
    &(m_L,\uu_L)|_{t=0}=(m_0,\uu_0).
  \end{aligned}\right.
  \end{align}

  {\rm (ii)} Uniqueness: Uniqueness holds in $E^{d/2}$ if $d\gs 3$. If $d=2$, one should also suppose that $n_0,m_0\in B^{\eps,1+\eps}$ and $\uu_0\in B^{\eps}$ for a $\eps\in(0,1)$, to get uniqueness in $E^1$.
\end{theorem}

With the help of  Theorem \ref{thm.2},  we can prove
Theorem~\ref{thm.1} as follows.

{\em Proof of Theorem \ref{thm.1}}.
From the conditions, we
have $n_0\in B^{d/2-1,d/2}$. In addition, from $\tilde{m}_0-\mb\in
B^{d/2-1,d/2}$, we can derive $m_0=a(\tilde{m}_0-\mb)/\mb+b
n_0\in B^{d/2-1,d/2}$. Since \eqref{eq.thm.1} implies
\eqref{thm.2.a}, the conclusion of Theorem \ref{thm.2} follows for
$(m,n,\uu)$. Changing back to the original variables
$(\tilde{m},\tilde{n},\uu)$, it leads to Theorem \ref{thm.1}. By Lemma \ref{lem.fgbesov}, it is easy to see that $\tilde{m}-\mb$ and $\tilde{n}-\nb$ also belong to $\mathcal{C}(\R^+;B^{d/2-1}\cap B^{d/2})$. \endproof

\subsection{{A priori} estimates for linear system with convection terms}\label{subsec.gwp2}

We first investigate some \emph{a priori} estimates for the linear system with convection terms
\begin{align}\label{eq.reform.linear}
  \left\{\begin{aligned}
    &n_t+\vv\cdot\nabla n=0,\\
    &m_t+\vv\cdot\nabla m+a\mb\dive\uu=F,\\
    &\uu_t+\vv\cdot\nabla\uu-\mu\Delta\uu-(\mu+\lambda)\nabla\dive\uu+C_0\nabla m=\GG,\\
    &(m,n,\uu)|_{t=0}=(m_0,n_0,\uu_0).
  \end{aligned}\right.
\end{align}
We do not need to separate the velocity into the compressible and
incompressible parts. In fact, we can prove the following
proposition.

\begin{proposition}\label{proposition}
  Let $a>0$, $\mb>0$, $s\in(1-d/2,d/2+1]$ and $s_1$, $s_2\in(-d/2,d/2+1]$ be constants. Assume $\vv\in L^1([0,T];B^{d/2+1})$ and denote $V(t)=\int_0^t\norm{\vv(\tau)}_{B^{d/2+1}}d\tau$. Let $(m,n,\uu)$ be a solution of \eqref{eq.reform.linear} on $[0,T]$, then the following estimates hold:
  \begin{align}\label{eq.prop.1}
    \norm{n}_{\LL([0,T]; B^{s_1,s_2})}\ls e^{CV(T)}\norm{n_0}_{B^{s_1,s_2}},
  \end{align}
  and
  \begin{align}\label{eq.prop.2}
  \begin{aligned}
    &\norm{m}_{\LL([0,T]; B^{s-1,s})}+\norm{\uu}_{\LL([0,T]; B^{s-1})}\\
    &\qquad\qquad +\norm{m}_{L^1([0,T]; B^{s+1,s})} +\norm{\uu}_{L^1([0,T]; B^{s+1})}\\
  \lesssim &e^{CV(T)} \left(\norm{m_0}_{B^{s-1,s}} +\norm{\uu_0}_{B^{s-1}}+ \norm{F}_{L^1([0,T]; B^{s-1,s})} +\norm{\GG}_{L^1([0,T]; B^{s-1})}\right).
  \end{aligned}
  \end{align}
\end{proposition}

\begin{proof}
\emph{Step 1: Estimates for the homogeneous transport equation.} We
derive the estimates for the first equation of
\eqref{eq.reform.linear} in Besov spaces.

Applying the Littlewood-Paley operator $\dk$ to the $\eqref{eq.reform.linear}_1$, it yields
\begin{align}\label{eq.transport}
  \left\{\begin{aligned}
    &\partial_t\dk[n]+\dk(\vv\cdot\nabla n)=0,\\
    &\dk[n]|_{t=0}=\dk n_0.
  \end{aligned}\right.
\end{align}

Taking the inner product of \eqref{eq.transport} with $\dk[n]$, we get for any $s_1,s_2\in (-d/2,1+d/2]$
\begin{align*}
  \frac{1}{2}\frac{d}{dt}\norm{\dk[n]}_2^2=-(\dk(\vv\cdot\nabla n),\dk[n])
  \lesssim\gamma_k2^{-k\varphi^{s_1,s_2}(k)}\norm{\vv}_{B^{d/2+1}} \norm{n}_{B^{s_1,s_2}}\norm{\dk[n]}_2.
\end{align*}
It follows
\begin{align*}
  \sum_k 2^{k\varphi^{s_1,s_2}(k)}\norm{\dk[n]}_2\ls \norm{n_0}_{B^{s_1,s_2}}+C\int_0^t \norm{\vv}_{B^{d/2+1}}\norm{n}_{B^{s_1,s_2}},
\end{align*}
which implies the desired estimate \eqref{eq.prop.1} with the help of the Gronwall inequality.

\emph{Step 2: Estimates for $(m,\uu)$.} Applying the
Littlewood-Paley operator $\dk$ to $\eqref{eq.reform.linear}_2$ and
$\eqref{eq.reform.linear}_3$, we have
\begin{align}\label{eq.ns}
  \left\{\begin{aligned}
    &\partial_t\dk[m]+\dk(\vv\cdot\nabla m)+a\mb\dive\dk[\uu]=\dk[F],\\
    &\partial_t\dk[\uu]+\dk(\vv\cdot\nabla\uu)-\mu\Delta\dk[\uu] -(\mu+\lambda)\nabla\dive\dk[\uu] +C_0\nabla \dk[m]=\dk[\GG].
  \end{aligned}\right.
\end{align}

Taking the inner product of $\eqref{eq.ns}_1$ with $\dk[m]$ and $-\Delta\dk[m]$, and $\eqref{eq.ns}_2$ with $\dk[\uu]$, we obtain
\begin{align}\label{eq.ns.1}
  \frac{1}{2}\frac{d}{dt}\norm{\dk[m]}_2^2&+(\dk(\vv\cdot\nabla m),\dk[m])+a\mb(\dive\dk[\uu],\dk[m])=(\dk[F],\dk[m]),\\
  \frac{1}{2}\frac{d}{dt}\norm{\nabla\dk[m]}_2^2&+(\dk(\vv\cdot\nabla m),-\Delta\dk[m])-a\mb(\dive\dk[\uu],\Delta\dk[m])\label{eq.ns.2}\\
  &\qquad=-(\dk[F],\Delta\dk[m]),\nonumber\\
  \frac{1}{2}\frac{d}{dt}\norm{\dk[\uu]}_2^2&+(\dk(\vv\cdot\nabla\uu),\dk[\uu]) +\mu\norm{\nabla\dk[\uu]}_2^2+(\mu+\lambda)\norm{\dive\dk[\uu]}_2^2 \label{eq.ns.3}\\
  &+C_0(\nabla \dk[m],\dk[\uu])=(\dk[\GG],\dk[\uu]).\nonumber
\end{align}
For the intersected term, we have
\begin{align}\label{eq.ns.4}
\begin{aligned}
  &\frac{d}{dt}(\dk[\uu],\nabla\dk[m])+(\dk(\vv\cdot\nabla\uu),\nabla\dk[m]) -(\dk(\vv\cdot\nabla m),\dive\dk[\uu])\\
    &\qquad-a\mb\norm{\dive\dk[\uu]}_2^2+C_0\norm{\nabla \dk[m]}_2^2+(2\mu+\lambda)(\dive\dk[\uu],\Delta\dk[m])\\
   =&-(\dk[F],\dive\dk[\uu])+(\dk[\GG],\nabla\dk[m]).
   \end{aligned}
\end{align}
Let
\begin{align}
  \alpha_k^2:=\frac{C_0}{a\mb}\norm{\dk[m]}_2^2
  +\norm{\dk[\uu]}_2^2 +\frac{(2\mu+\lambda)A}{a\mb}\norm{\nabla \dk[m]}_2^2+2A(\dk[\uu],\nabla\dk[m]).
\end{align}
For $A=(\mu+\lambda)/(2a\mb)>0$, there exist two positive constants $c_1$ and $c_2$ such that
\begin{align}
  c_1\alpha_k^2\ls \norm{\dk[m]}_2^2 +\norm{\dk[\uu]}_2^2+\norm{\nabla \dk[m]}_2^2\ls c_2\alpha_k^2,
\end{align}
since we have, for $M\in(a\mb/(2\mu+\lambda),2a\mb/(\mu+\lambda))$, that
\begin{align*}
  \abs{2(\dk[\uu],\nabla\dk[m])}\ls M\norm{\dk[\uu]}_2^2+\norm{\nabla\dk[m]}_2^2/M.
\end{align*}

Combining \eqref{eq.ns.1}-\eqref{eq.ns.4}, it yields, with the help of Lemma \ref{lem.innner}, that
\begin{align*}
  &\frac{1}{2}\frac{d}{dt}\alpha_k^2 +\mu\norm{\nabla \dk[\uu]}_2^2+(\mu+\lambda-a\mb A)\norm{\dive\dk[\uu]}_2^2
  +C_0A\norm{\nabla\dk[m]}_2^2\nonumber\\
  =&\frac{C_0}{a\mb}(\dk[F],\dk[m]) +(\dk[\GG],\dk[\uu])-\frac{(2\mu+\lambda)A}{a\mb}(\dk[F],\Delta\dk[m]) -A(\dk[F],\dive\dk[\uu])\nonumber\\
  &+A(\dk[\GG],\nabla\dk[m]) -\frac{C_0}{a\mb}(\dk(\vv\cdot\nabla m),\dk[m]) -(\dk(\vv\cdot\nabla\uu),\dk[\uu]) \nonumber\\
  &+\frac{(2\mu+\lambda)A}{a\mb}(\dk(\vv\cdot\nabla m),\Delta\dk[m])+A(\dk(\vv\cdot\nabla\uu),\nabla\dk[m])\\
  &+A(\nabla\dk(\vv\cdot\nabla m),\dk[\uu])
  \nonumber\\
  \lesssim &(\norm{\dk[F]}_2+\norm{\nabla\dk[F]}_2+\norm{\dk[\GG]}_2) (\norm{\dk[m]}_2 +\norm{\dk[\uu]}_2+\norm{\nabla \dk[m]}_2)\\
  &+\gamma_k2^{-k\varphi^{s-1,s}(k)}\norm{\vv}_{B^{d/2+1}} \norm{m}_{B^{s-1,s}}\norm{\dk[m]}_2\\
   &+\gamma_k2^{-k(s-1)}\norm{\vv}_{B^{d/2+1}} \norm{\uu}_{B^{s-1}}\norm{\dk[\uu]}_2\\
  &+\gamma_k2^{-k(\varphi^{s-1,s}(k)-1)}\norm{\vv}_{B^{d/2+1}} \norm{m}_{B^{s-1,s}}\norm{\nabla\dk[m]}_2\\
  &+\gamma_k\norm{\vv}_{B^{d/2+1}}\left(2^{-k(s-1)}\norm{\nabla\dk[m]}_2
  \norm{\uu}_{B^{s-1}}+2^{-k(\varphi^{s-1,s}(k)-1)} \norm{m}_{B^{s-1,s}} \norm{\dk[\uu]}_2\right)\\
  \lesssim &\Big(\norm{\dk[F]}_2+\norm{\nabla\dk[F]}_2+\norm{\dk[\GG]}_2 +\gamma_k2^{-k(s-1)}\norm{\vv}_{B^{d/2+1}}( \norm{m}_{B^{s-1,s}}+\norm{\uu}_{B^{s-1}})\Big) \\
  &\quad\times(\norm{\dk[m]}_2 +\norm{\dk[\uu]}_2+\norm{\nabla \dk[m]}_2).
\end{align*}
Thus, it follows
\begin{align*}
  &\frac{1}{2}\frac{d}{dt}\alpha_k^2+c_0\min(2^{2k},1)\alpha_k^2\\
  \lesssim& \gamma_k2^{-k(s-1)}\Big[\norm{F}_{B^{s-1,s}} +\norm{\GG}_{B^{s-1}}+\norm{\vv}_{B^{d/2+1}}\Big( \norm{m}_{B^{s-1,s}}+\norm{\uu}_{B^{s-1}}\Big)\Big]\alpha_k,
\end{align*}
which implies
\begin{align*}
  &2^{k(s-1)}\alpha_k+c_0\int_0^t\min(2^{2k},1)2^{k(s-1)}\alpha_k(\tau)d\tau\\
  \ls &2^{k(s-1)}\alpha_k(0)+C\gamma_k\int_0^t\left[\norm{F}_{B^{s-1,s}} +\norm{\GG}_{B^{s-1}}+\norm{\vv}_{B^{d/2+1}} \sum_k2^{k(s-1)}\alpha_k\right].
\end{align*}
Thus, by the Gronwall inequality, we have
\begin{align}\label{eq.ns.5}
\begin{aligned}
  &\norm{m}_{\LL([0,T]; B^{s-1,s})}+\norm{\uu}_{\LL([0,T]; B^{s-1})}\\
   &\qquad\qquad+\norm{m}_{L^1([0,T]; B^{s+1,s})} +\norm{\uu}_{L^1([0,T]; B^{s+1,s-1})}\\
  \lesssim &e^{CV(T)} \left(\norm{m_0}_{B^{s-1,s}} +\norm{\uu_0}_{B^{s-1}}+ \norm{F}_{L^1([0,T]; B^{s-1,s})} +\norm{\GG}_{L^1([0,T]; B^{s-1})}\right).
\end{aligned}
\end{align}

\emph{Step 3: The smoothing effect for $\uu$.} By \eqref{eq.ns.3},
we have
\begin{align*}
  &\frac{1}{2}\frac{d}{dt}\norm{\dk[\uu]}_2^2+C2^{2k}\norm{\dk[\uu]}_2^2\\
  \lesssim& \norm{\dk[\uu]}_2(2^k\norm{\dk[m]}_2+\norm{\dk[\GG]}_2+\gamma_k 2^{-k(s-1)} \norm{\vv}_{B^{d/2+1}}\norm{\uu}_{B^{s-1}}).
\end{align*}
It follows that
\begin{align*}
  &\frac{d}{dt}\sum_{k\gs 0}2^{k(s-1)}\norm{\dk[\uu]}_2 +C\sum_{k\gs 0}2^{k(s+1)}\norm{\dk[\uu]}_2\\
  \lesssim & \sum_{k\gs 0}2^{k(s-1)}\left[2^k\norm{\dk[m]}_2+\norm{\dk[\GG]}_2+\gamma_k 2^{-k(s-1)} \norm{\uu}_{B^{d/2+1}}\norm{\uu}_{B^{s-1}}\right]\\
  \lesssim & \sum_{k\gs 0}2^{ks}\norm{\dk[m]}_2 +\norm{\GG}_{B^{s-1}} +\norm{\vv}_{B^{d/2+1}}\norm{\uu}_{B^{s-1}},
\end{align*}
which implies, with the help of \eqref{eq.ns.5}, that
\begin{align*}
  &\int_0^t\sum_{k\gs 0}2^{k(s+1)}\norm{\dk[\uu](\tau)}_2\\
  \lesssim &\norm{\uu_0}_{B^{s-1}} +\int_0^t\sum_{k\gs 0}2^{ks}\norm{\dk[m](\tau)}_2d\tau +\int_0^t\norm{\GG(\tau)}_{B^{s-1}} d\tau\\
  &+\sup_{\tau\in[0,t]} \norm{\uu(\tau)}_{B^{s-1}}\int_0^t\norm{\vv(\tau)}_{B^{d/2+1}}d\tau\\
  \lesssim &e^{CV(t)} \left(\norm{m_0}_{B^{s-1,s}}+\norm{\uu_0}_{B^{s-1}}\int_0^t \left[\norm{F(\tau)}_{B^{s-1,s}} +\norm{\GG(\tau)}_{B^{s-1}}\right]d\tau\right).
\end{align*}
Combining with \eqref{eq.ns.5}, we get \eqref{eq.prop.2}.
\end{proof}

From the proof of Proposition \ref{proposition}, we immediately have

\begin{corollary}\label{coro}
  If a bounded operator $\mathcal{B}$ acts on the convection terms in  \eqref{eq.reform.linear}, then the same estimates hold for the refined system
  \begin{align}\label{eq.reform.linearB}
  \left\{\begin{aligned}
    &n_t+\mathcal{B}(\vv\cdot\nabla n)=0,\\
    &m_t+\mathcal{B}(\vv\cdot\nabla m)+a\mb\dive\uu=F,\\
    &\uu_t+\mathcal{B}(\vv\cdot\nabla\uu)-\mu\Delta\uu-(\mu+\lambda)\nabla\dive\uu+C_0\nabla m=\GG.
  \end{aligned}\right.
\end{align}
\end{corollary}

\subsection{Global existence and uniqueness of the solution}\label{subsec.gwp3}

\emph{Step 1: Friedrich's approximation.} Let $L_\ell^2$ be the
set of $L^2$ functions spectrally supported in the annulus
$\mathcal{C}_\ell:=\{\xi\in\R^d:\, 1/\ell\ls\abs{\xi}\ls \ell\}$
endowed with the standard $L^2$ topology. In order to construct the
classical Friedrichs approximation, we first define the Friedrichs
projectors $(\J)_{\ell\in\N}$ by
\begin{align*}
  \J f:=\F^{-1}\mathbf{1}_{\mathcal{C}_\ell}(\xi)\F f,
\end{align*}
for any $f\in L^2(\R^d)$ where $\mathbf{1}_{\mathcal{C}_\ell}(\xi)$ denotes the characteristic function on the annulus $\mathcal{C}_\ell$. Then, we can define the following approximate system
\begin{align}\label{eq.system.app}
  \left\{\begin{aligned}
    &n_t^\ell+\J(\J\uu^\ell\cdot\nabla \J n^\ell)=0,\\
    &m_t^\ell+\J(\J\uu^\ell\cdot\nabla\J m^\ell)+a\mb\dive\J \uu^\ell=F^\ell,\\
    &\uu_t^\ell+\J(\J\uu^\ell\cdot\nabla\J\uu^\ell)-\mu\Delta\J\uu^\ell -(\mu+\lambda)\nabla\dive\J\uu^\ell+C_0\nabla \J m^\ell=\GG^\ell,\\
    &(m^\ell,n^\ell,\uu^\ell)|_{t=0}=(m_\ell,n_\ell,\uu_\ell),
  \end{aligned}\right.
\end{align}
where
\begin{align*}
   m_\ell=&\J m_0, \quad n_\ell=\J n_0, \quad \uu_\ell=\J\uu_0, \\
   F^\ell=&\J F(\J m^\ell,\J n^\ell,\J\uu^\ell),\quad
   \GG^\ell=\J \GG(\J m^\ell,\J n^\ell,\J\uu^\ell).
\end{align*}

It is easy to check that it is an ordinary differential equation in $L_\ell^2\times L_\ell^2 \times (L_\ell^2)^d$ for every $\ell\in\N$. By the usual Cauchy-Lipschitz theorem, there is a strictly positive maximal time $T_\ell^*$ such that a unique solution $(m^\ell,n^\ell,\uu^\ell)$ exists in $[0,T_\ell^*)$ which is continuous in time with value in $L_\ell^2\times L_\ell^2 \times (L_\ell^2)^d$, i.e. $(m^\ell,n^\ell,\uu^\ell)\in \mathcal{C}([0,T_\ell^*); L_\ell^2\times L_\ell^2 \times (L_\ell^2)^d)$. As $\J^2=\J$, we see that $\J(m^\ell,n^\ell,\uu^\ell)$ is also a solution, so the uniqueness implies that $\J(m^\ell,n^\ell,\uu^\ell)=(m^\ell,n^\ell,\uu^\ell)$. Thus, this system can be rewritten as the following system
\begin{align}\label{eq.system.app1}
  \left\{\begin{aligned}
    &n_t^\ell+\J(\uu^\ell\cdot\nabla  n^\ell)=0,\\
    &m_t^\ell+\J(\uu^\ell\cdot\nabla m^\ell)+a\mb\dive \uu^\ell=F_1^\ell,\\
    &\uu_t^\ell+\J(\uu^\ell\cdot\nabla\uu^\ell)-\mu\Delta\uu^\ell -(\mu+\lambda)\nabla\dive\uu^\ell +C_0\nabla  m^\ell=\GG_1^\ell,\\
    &(m^\ell,n^\ell,\uu^\ell)|_{t=0}=(m_\ell,n_\ell,\uu_\ell),
  \end{aligned}\right.
\end{align}
where
\begin{align*}
   F_1^\ell=\J F( m^\ell, n^\ell,\uu^\ell), \text{ and }
   \GG_1^\ell=\J \GG( m^\ell,n^\ell,\uu^\ell).
\end{align*}

\emph{Step 2: Uniform estimates.} Denote
\begin{align*}
  E_0=\norm{m_0}_{B^{d/2-1,d/2}}+\norm{n_0}_{B^{d/2}}+\norm{\uu_0}_{B^{d/2-1}},
\end{align*}
and
\begin{align*}
  T_\ell:=\sup\{&T\in[0,T_\ell^*):
  \norm{(m^\ell,n^\ell,\uu^\ell)}_{E_T^{d/2}}\ls A \bar{C}E_0\},
\end{align*}
where $\bar{C}$ corresponds to the constant in Proposition \ref{proposition} and $A>\max(2,1/\bar{C})$ is a constant. Thus, by the continuity, we have $T_\ell>0$.

Let $M_0$ be the continuity modulus of the embedding relation $B^{d/2}(\R^d) \hookrightarrow L^\infty(\R^d)$. We make the assumption
\begin{align*}
  2(1+b)A\bar{C}M_0E_0\ls a\mb.
\end{align*}
Then, it implies
\begin{align*}
  \norm{m^\ell}_{L^\infty([0,T]\times\R^d)}\ls& M_0\norm{m^\ell}_{L^\infty([0,T];B^{d/2})}
  \ls M_0\norm{m^\ell}_{L^\infty([0,T];B^{d/2-1,d/2})}\\
  \ls &A\bar{C}M_0E_0
  \ls \frac{a\mb}{2(1+b)}.
\end{align*}
Similarly, we have
\begin{align*}
  \norm{n^\ell}_{L^\infty([0,T]\times\R^d)}\ls \frac{a\mb}{2(1+b)}.
\end{align*}
Then,
\begin{align*}
 \tilde{m}(m^\ell,n^\ell)=\mb+\frac{m-bn}{a} \in\left[\frac{\mb}{2},\frac{3\mb}{2}\right].
\end{align*}

By Proposition \ref{proposition}, we have
\begin{align*}
  &\norm{(m^\ell,n^\ell,\uu^\ell)}_{S_T}\\
\lesssim &e^{C\norm{\uu^\ell}_{L^1([0,T];B^{d/2+1})}} \Big(E_0+\norm{F_1^\ell}_{L^1([0,T];B^{d/2-1,d/2})} +\norm{\GG_1^\ell}_{L^1([0,T]; B^{d/2-1})}\Big).
\end{align*}

From Lemmas \ref{lem.comp} and \ref{lem.fgbesov}, we get
\begin{align*}
  &\norm{F_1^\ell}_{L^1([0,T];B^{d/2-1,d/2})}\\
  \lesssim &\norm{m^\ell\dive\uu^\ell}_{L^1([0,T];B^{d/2-1,d/2})} +\norm{n^\ell\dive\uu^\ell}_{L^1([0,T];B^{d/2-1,d/2})}\\
  \lesssim &(\norm{m^\ell}_{L^\infty([0,T]; B^{d/2-1,d/2})} +\norm{n^\ell}_{L^\infty([0,T]; B^{d/2-1,d/2})}) \norm{\dive\uu^\ell}_{L^1([0,T];B^{d/2})}\\
  \lesssim &(\norm{m^\ell}_{L^\infty([0,T]; B^{d/2-1,d/2})} +\norm{n^\ell}_{L^\infty([0,T]; B^{d/2-1,d/2})}) \norm{\uu^\ell}_{L^1([0,T];B^{d/2+1})}\\
  \lesssim &\norm{(m^\ell,n^\ell,\uu^\ell)}_{E_T^{d/2}}^2.
\end{align*}
Thus,
\begin{align*}
  &\norm{\frac{m^\ell-bn^\ell}{a\tilde{m}(m^\ell,n^\ell)} (\mu\Delta\uu^\ell+(\mu+\lambda)\nabla\dive\uu^\ell) }_{L^1([0,T];B^{d/2-1})}\\
  \lesssim &(\norm{n^\ell}_{L^\infty([0,T];B^{d/2})}+\norm{m^\ell}_{L^\infty([0,T];B^{d/2})})\norm{\uu^\ell}_{L^1([0,T];B^{d/2+1})}\\
  \lesssim &\norm{(m^\ell,n^\ell,\uu^\ell)}_{S_T}^2.
\end{align*}
Similarly, we can get
\begin{align*}
  \norm{\HH(m^\ell,n^\ell)}_{L^1([0,T];B^{d/2-1})}\lesssim & (\norm{n^\ell}_{L^\infty([0,T];B^{d/2})}+\norm{m^\ell}_{L^\infty([0,T];B^{d/2})})^2\\
  \lesssim&\norm{(m^\ell,n^\ell,\uu^\ell)}_{E_T^{d/2}}^2.
\end{align*}
Hence
\begin{align*}
  \norm{(m^\ell,n^\ell,\uu^\ell)}_{E_T^{d/2}}\ls &\bar{C}e^{\bar{C}\norm{(m^\ell,n^\ell,\uu^\ell)}_{E_T^{d/2}}} \Big(E_0+C\norm{(m^\ell,n^\ell,\uu^\ell)}_{E_T^{d/2}}^2\Big)\\
  \ls&\bar{C}e^{\bar{C}^2AE_0} (1+CA^2\bar{C}^2 E_0)E_0.
\end{align*}
Thus, we can choose $E_0$ so small that
\begin{align}\label{eq.small}
  1+CA^2\bar{C}^2E_0\ls \frac{A^2}{A+2}, \quad e^{\bar{C}^2AE_0}\ls \frac{A+1}{A}\quad \text{and } 2(1+b)A\bar{C}M_0E_0\ls a\mb,
\end{align}
which yields $\norm{(m^\ell,n^\ell,\uu^\ell)}_{E_T^{d/2}}\ls \frac{A+1}{A+2}A\bar{C}E_0$ for any $T<T_\ell$.

We claim that $T_\ell=T_\ell^*$. Indeed, if $T_\ell<T_\ell^*$, we
have seen that $\norm{(m^\ell,n^\ell,\uu^\ell)}_{E_T^{d/2}}\ls
\frac{A+1}{A+2}A\bar{C}E_0$. So by the continuity, for a sufficiently small
constant $s>0$, we can obtain
$\norm{(m^\ell,n^\ell,\uu^\ell)}_{E_{(T+s)}^{d/2}}\ls A\bar{C}E_0$ which
contradicts with the definition of $T_\ell$.

Now, we show the approximate solution is a global one, i.e. $T_\ell^*=\infty$. We assume $T_\ell^*<\infty$, then we have shown $\norm{(m^\ell,n^\ell,\uu^\ell)}_{E_{T}^{d/2}}\ls A\bar{C}E_0$. As
$$m^\ell\in L^\infty([0,T_\ell^*);B^{d/2-1,d/2}), n^\ell\in L^\infty([0,T_\ell^*);B^{d/2}) \text{ and } \uu^\ell\in L^\infty([0,T_\ell^*);B^{d/2-1}),$$
 it implies that
$$\norm{(m^\ell,n^\ell,\uu^\ell)}_{L^\infty([0,T_\ell^*);L_\ell^2)}<\infty.$$
Thus, we may extend the solution continuously beyond the time
$T_\ell^*$ by the Cauchy-Lipschitz theorem. This contradicts the
definition of $T_\ell^*$. Therefore, the solution
$(m^\ell,n^\ell,\uu^\ell)_{\ell\in\N}$ exists global in time.

\emph{Step 3: Time derivatives.} For convenience, we split the
approximate solution $(m^\ell,n^\ell,\uu^\ell)$ into a solution of
the linear system with initial data $(m_\ell,n_\ell,\uu_\ell)$, and
the discrepancy to that linear solution. More precisely, we denote
by $(m_L^\ell,\uu_L^\ell)$ the solution to the linear system
\begin{align}\label{eq.system.app2}
  \left\{\begin{aligned}
    &m_t^\ell+a\mb\dive \uu^\ell=0,\\
    &\uu_t^\ell-\mu\Delta\uu^\ell -(\mu+\lambda)\nabla\dive\uu^\ell +C_0\nabla  m^\ell=0,\\
    &(m^\ell,\uu^\ell)|_{t=0}=(m_\ell,\uu_\ell),
  \end{aligned}\right.
\end{align}
and $(m_D^\ell,n_D^\ell,\uu_D^\ell)=(m^\ell-m_L^\ell,n^\ell-n_\ell,\uu^\ell-\uu_L^\ell)$.

It is clear that the definition of $(m_\ell,n_\ell,\uu_\ell)$ implies
\begin{align*}
  m_\ell\to m_0 \text{ in } B^{d/2-1,d/2}, \quad n_\ell\to n_0 \text{ in } B^{d/2-1,d/2}, \quad \uu_\ell\to \uu_0 \text{ in } B^{d/2-1}.
\end{align*}
From Corollary \ref{coro}, we have
\begin{align*}
  (m_L^\ell,n_\ell,\uu_L^\ell)\to ( m_L,n_0,\uu_L) \text{ in } E^{d/2},
\end{align*}
where $m_L$ and $\uu_L$ satisfy the linear system \eqref{eq.system.app3}.

Now, we derive the uniform boundedness of the time derivatives of the discrepancy $(m_D^\ell,n_D^\ell,\uu_D^\ell)$.

\begin{lemma}\label{lem.time}
  $((m_D^\ell,n_D^\ell,\uu_D^\ell))_{\ell\in\N}$ is uniformly bounded in
  $$(\mathcal{C}^{1/2}(\R^+;B^{d/2-1}))^{1+1}\times (\mathcal{C}^{1/8}(\R^+; B^{d/2-5/4}))^d.$$
\end{lemma}

\begin{proof}
Since
  \begin{align*}
    \partial_t n_D^\ell=-\J(\uu^\ell\cdot\nabla n^\ell),
  \end{align*}
we have $\partial_t n_D^\ell\in L^2(\R^+;B^{d/2-1})$ since $n^\ell\in L^\infty(\R^+;B^{d/2})$ and $\uu^\ell\in L^2(\R^+;B^{d/2})$ with the help of the interpolation theorem.

From the equation
\begin{align*}
  \partial_t m_D^\ell=-\J(\uu^\ell\cdot\nabla m^\ell)-a\mb\dive\uu^\ell+a\mb\dive \uu_L^\ell- \J((m^\ell-bn^\ell)\dive\uu^\ell),
\end{align*}
it follows that $\partial_t m_D^\ell\in L^2(\R^+;B^{d/2-1})$.

Recall that
\begin{align*}
  \partial_t \uu_D^\ell=&-\J(\uu^\ell\cdot\nabla\uu^\ell)+\mu\Delta\uu^\ell+\mu\Delta\uu_L^\ell +(\mu+\lambda)\nabla\dive\uu^\ell-(\mu+\lambda)\nabla\dive\uu_L^\ell\\
  &-C_0\nabla m^\ell-C_0\nabla m_L^\ell+\GG_1^\ell,
\end{align*}
we can obtain $\partial_t \uu_D^\ell\in (L^\infty+L^{8/3}+L^{8/7})(\R^+;B^{d/2-5/4})$ through easy but tedious computations with the help of Lemmas \ref{lem.comp}, \ref{lem.fgbesov} and \ref{lem.inter}.

Applying the Morrey embedding relation $W^{1,p}(\R)\subset C^{1-1/p}(\R)$ to the time variable for $1<p\ls\infty$, we obtain the desired result.
\end{proof}

\emph{Step 4: Compactness and convergence.} The proof of the
existence of a solution is now standard. Indeed, we can use
Arzel\`a-Ascoli theorem to get strong convergence of the approximate
solutions. We need to localize the spatial space in order to utilize
some compactness results of local Besov spaces (see \cite[Chapter
2]{BCDbook}). Let $(\chi_p)_{p\in\N}$ be a sequence of
$\mathscr{D}(\R^d)$ cut-off functions supported in the ball
$B(0,p+1)$ of $\R^d$ and equal to $1$ in a neighborhood of $B(0,p)$.
In view of Lemma \ref{lem.time} and uniform estimates obtained in
Step 2, we see that
$((\chi_pm_D^\ell,\chi_pn_D^\ell,\chi_p\uu_D^\ell))_{\ell\in\N}$ is
bounded in $E^{d/2}$ and uniformly equi-continuous in
$$\mathcal{C}\left([0,T]; (B^{d/2-1})^{1+1}\times (B^{d/2-5/4})^d\right)$$
for
any $p\in\N$ and $T>0$. Moreover, the mapping $f\mapsto \chi_p f$ is
compact from $B^{d/2-1,d/2}$ into $B^{d/2-1}$ and from $B^{d/2-1}$
into $B^{d/2-5/4}$.

Applying the Arzel\`a-Ascoli theorem to the family
$((\chi_pm_D^\ell,\chi_pn_D^\ell,\chi_p\uu_D^\ell))_{\ell\in\N}$ on
the time interval $[0,p]$, then we use the Cantor diagonal process.
This finally provides us with a distribution $(m_D,n_D,\uu_D)$
continuous in time with values in $(B^{d/2-1})^{1+1}\times
(B^{d/2-5/4})^d$ and a subsequence (which we still denote by
the same notation) such that we have
for all $p\in\N$
$$
(\chi_pm_D^\ell,\chi_pn_D^\ell,\chi_p\uu_D^\ell)
\to (\chi_pm_D,\chi_pn_D,\chi_p\uu_D), \text{ as } \ell\to\infty
$$
in $\mathcal{C}([0,p]; (B^{d/2-1})^{1+1}\times (B^{d/2-5/4})^d)$.
This obviously implies that $(m_D^\ell,n_D^\ell,\uu_D^\ell)$ tends
to $(m_D,n_D,\uu_D)$ in $\mathscr{D}'(\R^+\times\R^d)$.

Coming back to the uniform estimates and Lemma \ref{lem.time}, we
further  obtain that $(m_D,n_D,\uu_D)$ belongs to $E^{d/2}$ and to
$(\mathcal{C}^{1/2}(\R^+;B^{d/2-1}))^{1+1}\times
(\mathcal{C}^{1/8}(\R^+; B^{d/2-5/4}))^d$. The convergence results
stemming from this last result and the interpolation argument enable
us to pass to the limit in $\mathscr{D}'(\R^+\times\R^d)$ in the
system \eqref{eq.system.app} and to prove that
$(m,n,\uu):=(m_L,n_L,\uu_L)+(m_D,n_D,\uu_D)$ is indeed a solution of
\eqref{eq.reform.system} with the initial data. Since it is just a
matter of doing tedious verifications, we omit the details.

\emph{Step 5: Continuities in time.} The continuity of $\uu$ is
straightforward. Indeed, from the third equation of
\eqref{eq.reform.system}, we have $\uu_t\in
(L^1+L^2)(\R^+;(B^{d/2-1})^d)$ which implies
$\uu\in\mathcal{C}(\R^+;(B^{d/2-1})^d)$ in view of the Morrey
embedding and the embedding relation $W^{1,1}(\R)\subset
\mathcal{C}(\R)$. Consequently, the continuity of $n$ in time is
obtained from \eqref{eq.prop.1}. For $m$, it is easily to see that
$m_t\in L^2(\R^+;B^{d/2-1})\cap L^1(\R^+;B^{d/2})$ from the second
equation of \eqref{eq.reform.system} which yields $m\in
\mathcal{C}(\R^+;B^{d/2-1,d/2})$ by the embeddings mentioned above.

\emph{Step 6: Uniqueness.} Next, we prove the uniqueness of
solutions. Let $(m_1,n_1,\uu_1)$ and $(m_2,n_2,\uu_2)$ be two
solutions of \eqref{eq.reform.system} in $E_T^{d/2}$ with the same
initial data. Denote $(\delta m,\delta
n,\delta\uu)=(m_2-m_1,n_2-n_1,\uu_2-\uu_1)$. Then they satisfy
\begin{align}
  \left\{\begin{aligned}
    &\partial_t\delta n+\uu_2\cdot\nabla \delta n=-\delta\uu\cdot\nabla n_1,\\
    &\partial_t\delta m+\uu_2\cdot\nabla \delta m+a\mb\dive\delta\uu=-\delta\uu\cdot\nabla m_1+\delta F,\\
    &\partial_t\delta\uu+\uu_2\cdot\nabla\delta\uu-\mu\Delta\delta\uu -(\mu+\lambda)\nabla\dive\delta\uu+C_0\nabla\delta m=-\delta\uu\cdot\nabla \uu_1+\delta\GG,\\
    &(\delta m,\delta n,\delta\uu)|_{t=0}=(0,0,\mathbf{0}),
  \end{aligned}\right.
\end{align}
where $\delta F=F(m_2,n_2,\uu_2)-F(m_1,n_1,\uu_1)$ and $\delta \GG=\GG(m_2,n_2,\uu_2)-\GG(m_1,n_1,\uu_1)$.

We first consider the case $d\gs 3$. Similar to the derivation of \eqref{eq.prop.1}, we can get for $t\in[0,T]$ with the help of \eqref{eq.prop.1}
\begin{align}
  \norm{\delta n}_{\LL([0,T]; B^{d/2-2,d/2-1})}\ls & e^{C\int_0^T\norm{\uu_2}_{B^{d/2+1}}d\tau}\norm{\delta\uu}_{L^1([0,T];B^{d/2})} \norm{n_1}_{L^\infty([0,T];B^{d/2-1,d/2})}\nonumber\\
  \lesssim& e^{C\int_0^T\norm{(\uu_1,\uu_2)}_{B^{d/2+1}}d\tau}\norm{n_0}_{B^{d/2-1,d/2}} \norm{\delta\uu}_{L^1([0,T];B^{d/2})} .
\end{align}
By Lemma \ref{lem.comp} and \ref{lem.fgbesov}, we have
\begin{align*}
   &\norm{\delta m}_{\LL([0,T]; B^{d/2-2,d/2-1})}+\norm{\delta\uu}_{\LL([0,T];B^{d/2-2})}\\
   &\qquad\qquad+\norm{\delta m}_{L^1([0,T]; B^{d/2,d/2-1})}+\norm{\delta\uu}_{L^1([0,T];B^{d/2})}\nonumber\\
  \lesssim &e^{C\int_0^T\norm{\uu_2}_{B^{d/2+1}}d\tau} \big[\norm{\delta\uu\cdot\nabla m_1}_{L^1([0,T]; B^{d/2-2,d/2-1})} +\norm{\delta F(\tau)}_{L^1([0,T]; B^{d/2-2,d/2-1})}\\
   &\qquad\qquad+\norm{\delta\uu\cdot\nabla\uu_1}_{L^1([0,T]; B^{d/2-2})} +\norm{\delta\GG(\tau)}_{L^1([0,T]; B^{d/2-2})}\big]\\
   \lesssim & e^{C\int_0^T\norm{\uu_2}_{B^{d/2+1}}d\tau}(\norm{\delta\uu}_{L^1([0,T];B^{d/2})} \norm{m_1}_{\LL([0,T];B^{d/2-1,d/2})}\\
   &+\norm{(\delta m,\delta n)}_{\LL([0,T]; B^{d/2-2,d/2-1})} \norm{\uu_2}_{L^1([0,T]; B^{d/2+1})}\\
    &+\norm{(m_1,n_1)}_{\LL([0,T]; B^{d/2-1,d/2})}\norm{\delta\uu}_{L^1([0,T]; B^{d/2})}\\
   &+\norm{\delta\uu}_{\LL([0,T]; B^{d/2-2})}\norm{\uu_1}_{L^1([0,T]; B^{d/2+1})}+\norm{\delta\GG(\tau)}_{L^1([0,T]; B^{d/2-2})}\\
   \lesssim & e^{C\norm{\uu_2}_{L^1([0,T]; B^{d/2+1})}} \Big(\big(1+\norm{(m_2,n_2)}_{\LL([0,T]; B^{d/2-1,d/2})}\big)\\
   &\qquad\qquad\qquad\qquad \times\norm{(m_1,n_1)}_{\LL([0,T]; B^{d/2-1,d/2})}+Z(T)\Big)\norm{(\delta m,\delta n,\delta\uu)}_{S_T^{d/2-1}},
\end{align*}
where $\limsup_{T\to 0^+}Z(T)=0$. Thus,
\begin{align*}
  &\norm{(\delta m,\delta n,\delta\uu)}_{E_T^{d/2-1}}\\
  \ls& C e^{C\norm{(\uu_1,\uu_2)}_{L^1([0,T]; B^{d/2+1})}} \Big((1+\norm{(m_2,n_2)}_{\LL([0,T]; B^{d/2-1,d/2})})\\
  &\times\norm{(m_1,n_1)}_{\LL([0,T]; B^{d/2-1,d/2})}+E_0+Z(T)\Big)\norm{(\delta m,\delta n,\delta\uu)}_{E_T^{d/2-1}}.
\end{align*}
We take $E_0$ small enough such that it satisfies the condition $2C(1+CA\bar{C}E_0)A\bar{C}E_0+E_0<1/4$ and  \eqref{eq.small}, and choose $T>0$ so small that $C\norm{(\uu_1,\uu_2)}_{L^1([0,T]; B^{d/2+1})}\ls \ln 2$ and $Z(T)<1/2$, then it follows $\norm{(\delta m,\delta n,\delta\uu)}_{E_T^{d/2-1}}\equiv 0$. Hence, $(m_1,n_1,\uu_1)(t)=(m_2,n_2,\uu_2)(t)$ on $[0,T]$. By a standard argument (e.g. \cite{Danchin00}), we can conclude that $(m_1,n_1,\uu_1)(t)=(m_2,n_2,\uu_2)(t)$ on $\R^+$.

For the case $d=2$, we have to raise the regularity of the spaces.
Thus, we also suppose that $m_0,n_0\in B^{\eps,1+\eps}$ and
$\uu_0\in B^{\eps}$ for a $\eps\in(0,1)$. By the same process, we
can prove the existence of solution $(m,n,\uu)$ in the space
$E^{1+\eps}$  provided the norms of initial data is sufficiently
small. Then, in the same way as in the case $d\gs 3$, we may prove
the uniqueness of solutions in the space $E^{\eps}$ (of course,
holds in $E^1$). We omit the details.

\section{Local well-posedness for large data}\label{sec.lwp}

\subsection{Reformulation of the system}

We change variables to $\rho=\mb(\tilde{m}^{-1}-\mb^{-1})$ and
$g=\tilde{n}-\nb$. Then we can reformulate the system
\eqref{eq.system}-\eqref{eq.data} as
\begin{align}\label{eq.system.2}
  \left\{\begin{aligned}
    &\rho_t+\uu\cdot\nabla\rho=(\rho+1)\dive\uu,\\
    &g_t+\uu\cdot\nabla g=-(g+\nb)\dive\uu,\\
    &\uu_t+\uu\cdot\nabla\uu-(1+\rho)(\mu\Delta\uu+(\mu+\lambda)\nabla\dive\uu) +Q(\rho,g)=0,\\
    &(\rho,g,\uu)|_{t=0}=(\rho_0,g_0,\uu_0),
  \end{aligned}\right.
\end{align}
where $\rho_0=\mb(\tilde{m}_0^{-1}-\mb^{-1})$, $g_0=\tilde{n}_0-\nb$ and
\begin{align*}
  Q(\rho,g):=&\mb^{-1}(1+\rho)\nabla P(\mb/(1+\rho),g+\nb)\\
  =&\frac{\rho\nabla \rho}{\rho+1}-\nabla\rho+\frac{a_0}{\mb}(\rho+1)\nabla g+B(\rho,g)\Big[-\frac{\mb}{(\rho+1)^2}\nabla\rho +\frac{k_0-a_0\nb}{\rho+1}\nabla\rho\\
  &+\frac{a_0g\nabla\rho}{\rho+1} +\frac{a_0^2}{\mb}(g+g\rho)\nabla g
+\frac{a_0(k_0+\mb+a_0\nb)}{\mb}\nabla g+\frac{a_0k_0+a_0^2\nb}{\mb}\rho\nabla g\Big],
\end{align*}
with
\begin{align*}
  B(\rho,g):=\Big[\big(\frac{\mb}{\rho+1}+a_0(g+\nb)-k_0\big)^2+4k_0a_0(g+\nb)\Big]^{-1/2}.
\end{align*}

We now state the result for the local theory for general data bounded away from the vacuum as follows.

\begin{theorem}\label{thm.4}
  Let $d\gs 2$, $\mu>0$, $2\mu+d\lambda\gs 0$, the constants $\mb>0$ and $\nb\gs 0$. Assume that $\rho_0\in B^{d/2,d/2+1}$, $g_0\in B^{d/2,d/2+1}$ and $\uu_0\in B^{d/2-1,d/2}$. In addition, $\inf\limits_{x\in\R^d} \rho_0(x)>-1$. Then there exists a positive time $T$ such that the system \eqref{eq.system.2} has a unique solution $(\rho,g,\uu)$ on $[0,T]\times\R^d$ which belongs to
  $$(\tilde{\mathcal{C}}([0,T];B^{d/2,d/2+1}))^{1+1}\times(\tilde{\mathcal{C}}([0,T]; B^{d/2-1,d/2})\cap L^1([0,T];B^{d/2+1,d/2+2}))^d,$$
  and satisfies $\inf\limits_{(t,x)\in[0,T]\times\R^d} \rho(t,x)>-1$.
\end{theorem}

\subsection{A priori Estimates}

Now, let us recall some estimates for the following parabolic system which is obtained by linearizing the momentum equation
\begin{align*}
  \left\{\begin{aligned}
    &\uu_t+\vv\cdot\nabla\uu+\uu\cdot\nabla\ww-b(t,x)(\mu\Delta\uu+(\mu+\lambda) \nabla\dive\uu) =f,\\
    &\uu|_{t=0}=\uu_0,
  \end{aligned}\right.
\end{align*}
which had been studied in \cite{BCDbook,Danchin07}. Precisely, we have the following lemma (cf. \cite[Proposition 10.12]{BCDbook}).

\begin{lemma}\label{lem.moment}
  Let $\alpha\in(0,1]$, $s\in(-d/2,d/2]$, $\underline{\nu}=\min(\mu,\lambda+2\mu)$ and $\bar{\nu}=\mu+\abs{\mu+\lambda}$. Assume that $b=1+\rho$ with $\rho\in L^\infty([0,T]; B^{d/2+\alpha})$ and that
  \begin{align*}
    b_*:=\inf_{(t,x)\in[0,T]\times\R^d} b(t,x)>0.
  \end{align*}
  There exist a universal constant $\kappa$, and a constant $C$ depending only on $d$, $\alpha$ and $s$, such that for all $t\in[0,T]$,
  \begin{align*}
    &\norm{\uu}_{\LL([0,t];B^{s})}+\kappa b_*\underline{\nu}\norm{\uu}_{L^1([0,t];B^{s+2})}\\
    \ls &\left(\norm{\uu_0}_{B^{s}} +\norm{f}_{L^1([0,t];B^{s})}\right)\\
    &\times\exp\left(C\int_0^t\Big(\norm{\vv}_{B^{d/2+1}}+\norm{\ww}_{B^{d/2+1}} +(b_*\underline{\nu})^{1-2/\alpha}\bar{\nu}^{2/\alpha} \norm{\rho}_{B^{d/2+\alpha}}^{2/\alpha}\Big)d\tau\right).
  \end{align*}
  If $\vv$ and $\ww$ depend linearly on $\uu$, then the above inequality is true for all $s\in(0,d/2+\alpha]$, and the argument of the exponential term may be replaced with
  \begin{align*}
    C\int_0^t\left(\norm{\nabla\uu}_\infty+(b_*\underline{\nu})^{1-2/\alpha} \bar{\nu}^{2/\alpha}\norm{\rho}_{B^{d/2+\alpha}}^{2/\alpha}\right)d\tau.
  \end{align*}
\end{lemma}

For the mass equations, we only need to study the following equation
with two constants $\theta\in\R$ and $\beta>0$
\begin{align}\label{eq.trans}
  \left\{\begin{aligned}
    &h_t+\vv\cdot\nabla h=\theta(h+\beta)\dive\vv,\\
    &h|_{t=0}=h_0.
  \end{aligned}\right.
\end{align}
\begin{proposition}\label{prop.2}
  Let $s\in(-d/2,d/2+1]$, $T>0$, $\theta\in\R$ and $\beta\gs 0$ be constants. Assume that $h_0\in B^{d/2}$, $\vv\in L^1([0,T); B^{d/2+1})$ and $a$ satisfies \eqref{eq.trans}. There exists a constant $C$ depending only on $d$ such that for all $t\in[0,T]$, we have
  \begin{align}\label{eq.mass.1}
      \norm{h}_{\LL([0,t];B^{d/2})}\ls e^{C(1+2\abs{\theta})\int_0^t\norm{\vv}_{B^{d/2+1}}d\tau} \left(\norm{h_0}_{B^{d/2}}+\frac{\beta}{1+2\abs{\theta}}\right) -\frac{\beta}{1+2\abs{\theta}},
\end{align}
and
\begin{align} \label{eq.mass.1a}
\begin{aligned}
  &\norm{h}_{\LL([0,t];B^{s})}
    \ls e^{C(1+\abs{\theta})\int_0^t\norm{\vv}_{B^{d/2+1}}d\tau}\Big(\norm{h_0}_{B^{s}}
    \\ &+C\abs{\theta}\Big[e^{C(1+2\abs{\theta})\int_0^t\norm{\vv}_{B^{d/2+1}}d\tau} \Big(\norm{h_0}_{B^{d/2}}+\frac{\beta}{1+2\abs{\theta}}\Big) +\frac{2\abs{\theta}\beta}{1+2\abs{\theta}}\Big]\int_0^t\norm{\vv}_{B^{s+1}}d\tau\Big).
\end{aligned}
\end{align}
\end{proposition}

\begin{proof}
  Applying the operator $\dk$ to \eqref{eq.trans} yields
  \begin{align*}
    \partial_t\dk[h]+\dk(\vv\cdot\nabla h)=\theta\dk((h+\beta)\dive\vv).
  \end{align*}
  Taking the $L^2$ inner product with $\dk[h]$, we get, with the help of Lemmas \ref{lem.fgbesov} and \ref{lem.innner}, that
  \begin{align*}
    &\frac{1}{2}\frac{d}{dt}\norm{\dk[h]}_2^2\\
    =&-(\dk(\vv\cdot\nabla a),\dk[h])+\theta(\dk((h+\beta)\dive\vv),\dk[h])\\
    \lesssim&\gamma_k 2^{-ks}\norm{\vv}_{B^{d/2+1}}\norm{h}_{B^s}\norm{\dk[h]}_2 +\abs{\theta}\gamma_k 2^{-ks}(\norm{h\dive\vv}_{B^{s}}+\beta\norm{\dive\vv}_{B^{s}})\norm{\dk[h]}_2\\
    \lesssim& \gamma_k 2^{-ks}((1+\abs{\theta})\norm{h}_{B^{s}}\norm{\vv}_{B^{d/2+1}} +(\norm{h}_{B^{d/2}}+\beta)\abs{\theta}\norm{\vv}_{B^{s+1}})\norm{\dk[h]}_2.
  \end{align*}
Eliminating the factor $\norm{\dk[h]}_2$ from both sides and integrating in the time, we have
\begin{align*}
\norm{\dk[h]}_2\ls& \norm{\dk[h_0]}_2\\
&+C\gamma_k\int_0^t 2^{-ks}((1+\abs{\theta})\norm{h}_{B^{s}}\norm{\vv}_{B^{d/2+1}} +(\norm{h}_{B^{d/2}}+\beta)\abs{\theta}\norm{\vv}_{B^{s+1}})d\tau.
\end{align*}
It follows, for any $k\in\Z$ and any $t\in[0,T]$, that
\begin{align}\label{eq.mass.4}
\begin{aligned}
  &2^{ks}\norm{\dk[h]}_2\ls 2^{ks}\norm{\dk[h_0]}_2\\
  &\qquad\qquad+C\gamma_k\int_0^t( (1+\abs{\theta})\norm{h}_{B^{s}}\norm{\vv}_{B^{d/2+1}} +(\norm{h}_{B^{d/2}}+\beta)\abs{\theta}\norm{\vv}_{B^{s+1}})d\tau.
\end{aligned}
\end{align}
Summing up on $k\in \Z$, it yields
\begin{align*}
  \norm{h}_{\LL([0,T];B^{s})}\ls &\norm{h_0}_{B^{s}}\\
  &+\int_0^t C[(1+\abs{\theta})\norm{\vv}_{B^{d/2+1}}\norm{h}_{B^{s}} +(\norm{h}_{B^{d/2}}+\beta)\abs{\theta}\norm{\vv}_{B^{s+1}}]d\tau.
\end{align*}

By the Gronwall inequality, we have \eqref{eq.mass.1} for $s=d/2$ and then for any $s\in(-d/2,d/2+1]$
\begin{align*}
    &\norm{h}_{\LL([0,t];B^{s})}\\
    \ls& e^{C(1+\abs{\theta})\int_0^t\norm{\vv}_{B^{d/2+1}}d\tau} \left(\norm{h_0}_{B^{s}}+C\int_0^t\abs{\theta}(\norm{h}_{B^{d/2}}+\beta) \norm{\vv}_{B^{s+1}}d\tau\right)\\
    \ls&e^{C(1+\abs{\theta})\int_0^t\norm{\vv}_{B^{d/2+1}}d\tau}\Big(\norm{h_0}_{B^{s}} \\ &+C\abs{\theta}\left[e^{C(1+2\abs{\theta})\int_0^t\norm{\vv}_{B^{d/2+1}}d\tau} \Big(\norm{h_0}_{B^{d/2}}+\frac{\beta}{1+2\abs{\theta}}\Big) +\frac{2\abs{\theta}\beta}{1+2\abs{\theta}}\right]\int_0^t\norm{\vv}_{B^{s+1}}d\tau\Big).
\end{align*}
This completes the proofs.
\end{proof}

In general, for the transport equation
\begin{align*}
  \left\{\begin{aligned}
    &h_t+\vv\cdot\nabla h=f,\\
    &h(0)=h_0,
  \end{aligned}\right.
\end{align*}
we can get, in a similar way with Proposition \ref{prop.2}, that
\begin{proposition}\label{prop.3}
  Let $s_1,s_2\in(-d/2,d/2+1]$ and $T>0$. Then it holds for $t\in[0,T]$
  \begin{align*}
    \norm{h}_{\LL([0,t]; B^{s_1,s_2})}\ls e^{C\int_0^t\norm{\vv}_{B^{d/2+1}}d\tau}\left( \norm{h_0}_{B^{s_1,s_2}}+\int_0^t\norm{f}_{B^{s_1,s_2}}d\tau\right).
  \end{align*}
\end{proposition}

\subsection{Existence of local solution}

\emph{Step 1: The Friedrich's approximation.} For convenience, we
introduce the solution $\uu_{\mathrm{ls}}$ to the linear system
\begin{align}\label{eq.ls}
  \partial_t\uu_{\mathrm{ls}}-\mu\Delta\uu_{\mathrm{ls}}-(\mu+\lambda) \nabla\dive\uu_{\mathrm{ls}}=0, \quad \uu_{\mathrm{ls}}(0)=\uu_0.
\end{align}
Denote $\uu_{\mathrm{ls}}^\ell:=\J\uu_{\mathrm{ls}}$ and  $\tilde{\uu}^\ell:=\uu^\ell-\uu_{\mathrm{ls}}^\ell$. Then we can construct the following approximation $(\rho^\ell,g^\ell,\tilde{\uu}^\ell)$ satisfying
\begin{align}\label{eq.system.app4}
  \left\{\begin{aligned}
    &\rho_t^\ell+\J(\uu^\ell\cdot\nabla\rho^\ell)=\J((\rho^\ell+1)\dive\uu^\ell),\\
    &g_t^\ell+\J(\uu^\ell\cdot\nabla g^\ell)=-\J((g^\ell+\nb)\dive\uu^\ell),\\
    &\partial_t\tilde{\uu}^\ell+\J(\uu_{\mathrm{ls}}^\ell\cdot\nabla \tilde{\uu}^\ell)+\J(\tilde{\uu}^\ell\cdot\nabla \uu^\ell)-\J[(1+\rho^\ell)(\mu\Delta\tilde{\uu}^\ell +(\mu+\lambda)\nabla \dive\tilde{\uu}^\ell)]\\
    &\qquad=\J[\rho^\ell(\mu\Delta\uu_{\mathrm{ls}}^\ell +(\mu+\lambda)\nabla\dive\uu_{\mathrm{ls}}^\ell)] -\J(\uu_{\mathrm{ls}}^\ell\cdot\nabla \uu_{\mathrm{ls}}^\ell)-\J Q(\rho^\ell,g^\ell),\\
    &(\rho^\ell,g^\ell,\tilde{\uu}^\ell)|_{t=0}=(\rho_0^\ell,g_0^\ell,\mathbf{0}),
  \end{aligned}\right.
\end{align}
where $\rho_0^\ell:=\J\rho_0$, $g_0^\ell:=\J g_0$ and
$\uu^\ell:=\uu_{\mathrm{ls}}^\ell+\tilde{\uu}^\ell$.

Note that if $1+\rho_0$ is bounded away from zero, then so is $1+\J
\rho_0$ for sufficiently large $\ell$. It is easy to check that
\eqref{eq.system.app4} is an ordinary differential equation in
$L_\ell^2\times L_\ell^2 \times (L_\ell^2)^d$ for every $\ell\in\N$.
By the usual Cauchy-Lipschitz theorem, there is a strictly positive
maximal time $T_\ell^*$ such that a unique solution
$(\rho^\ell,g^\ell,\tilde{\uu}^\ell)$ exists in $[0,T_\ell^*)$ which
is continuous in time with value in $L_\ell^2\times L_\ell^2 \times
(L_\ell^2)^d$, i.e. $(\rho^\ell,g^\ell,\tilde{\uu}^\ell)\in
\mathcal{C}([0,T_\ell^*); L_\ell^2\times L_\ell^2 \times
(L_\ell^2)^d)$, and $1+\rho^\ell$ is bounded away from zero.

\emph{Step 2: Lower bound for lifespan and uniform estimates.} We
introduce the following notations
\begin{align*}
  M_0:=&\norm{\rho_0}_{B^{d/2,d/2+\alpha}},\quad M^\ell(t):=\norm{\rho^\ell}_{\LL([0,t];B^{d/2,d/2+\alpha})},\\
  N_0:=&\norm{g_0}_{B^{d/2,d/2+\alpha}},\quad N^\ell(t):=\norm{g^\ell}_{\LL([0,t];B^{d/2,d/2+\alpha})},\\
  U_0:=&\norm{\uu_0}_{B^{d/2-1,d/2-1+\alpha}},\quad U_{\mathrm{ls}}^\ell(t):=\norm{\uu_{\mathrm{ls}}^\ell}_{L^1([0,t]; B^{d/2+1,d/2+1+\alpha})},\\
  \tilde{U}^\ell(t):=&\norm{\tilde{\uu}^\ell}_{\LL([0,t];B^{d/2-1,d/2-1+\alpha})} +b_*\underline{\nu}\norm{\tilde{\uu}^\ell}_{L^1([0,t];B^{d/2+1,d/2+1+\alpha})}.
\end{align*}

In view of Lemma \ref{lem.moment}, we take $\vv=\ww=f=\mathbf{0}$ and $\rho=0$ there to get
\begin{align*}
  \norm{\uu_{\mathrm{ls}}^\ell}_{\LL([0,t];B^{d/2-1,d/2-1+\alpha})}\lesssim U_0.
\end{align*}

Note that for all $k\in\Z$, we have
\begin{align*}
  \norm{\dk \rho_0^\ell}_2\ls\norm{\dk \rho_0}_2,\quad  \norm{\rho_0^\ell}_{B^{s}}\ls\norm{\rho_0}_{B^{s}}
\end{align*}
and similar properties for $g_0^\ell$ because of the boundedness of the operators $\dk$.

From \eqref{eq.mass.1} and \eqref{eq.mass.1a}, we get
\begin{align}\label{eq.mass.5}
      \norm{\rho^\ell}_{\LL([0,t];B^{d/2})}\ls e^{3C\int_0^t\norm{\uu^\ell}_{B^{d/2+1}}d\tau} \left(\norm{\rho_0}_{B^{d/2}}+\frac{1}{3}\right) -\frac{1}{3},
\end{align}
\begin{align}  \label{eq.mass.6}
    &\norm{\rho^\ell}_{\LL([0,t];B^{d/2+\alpha})}
    \ls e^{2C\int_0^t\norm{\uu^\ell}_{B^{d/2+1}}d\tau} \Big(\norm{\rho_0}_{B^{d/2+\alpha}}\nonumber \\ &+C\left[e^{3C\int_0^t\norm{\uu^\ell}_{B^{d/2+1}}d\tau} \Big(\norm{\rho_0}_{B^{d/2}}+\frac{1}{3}\Big) +\frac{2}{3}\right] \int_0^t\norm{\uu^\ell}_{B^{d/2+1+\alpha}}d\tau\Big),
\end{align}
\begin{align}\label{eq.mass.7}
      \norm{g^\ell}_{\LL([0,t];B^{d/2})}\ls e^{3C\int_0^t\norm{\uu^\ell}_{B^{d/2+1}}d\tau} \left(\norm{g_0}_{B^{d/2}}+\frac{\nb}{3}\right) -\frac{\nb}{3},
\end{align}
and
\begin{align}  \label{eq.mass.8}
    &\norm{g^\ell}_{\LL([0,t];B^{d/2+\alpha})}
    \ls e^{2C\int_0^t\norm{\uu^\ell}_{B^{d/2+1}}d\tau} \Big(\norm{g_0}_{B^{d/2+\alpha}}\nonumber \\ &+C\left[e^{3C\int_0^t\norm{\uu^\ell}_{B^{d/2+1}}d\tau} \Big(\norm{g_0}_{B^{d/2}}+\frac{\nb}{3}\Big) +\frac{2\nb}{3}\right] \int_0^t\norm{\uu^\ell}_{B^{d/2+1+\alpha}}d\tau\Big).
\end{align}

Let $b(t,x)=1+\rho^\ell$. From Step 1, we know $b_*>0$. Thus, by Lemma \ref{lem.moment}, \ref{lem.comp} and \ref{lem.fgbesov}, we have
\begin{align}\label{eq.ns.9}
\begin{aligned}
  &\norm{\tilde{\uu}^\ell}_{\LL([0,t];B^{d/2-1})}+\kappa b_*\underline{\nu} \norm{\tilde{\uu}^\ell}_{L^1([0,t];B^{d/2+1})}\\
  \ls &\int_0^t\Big(\norm{\rho^\ell(\mu\Delta\uu_{\mathrm{ls}}^\ell +(\mu+\lambda)\nabla\dive\uu_{\mathrm{ls}}^\ell)}_{B^{d/2-1}} \\ &\qquad\qquad+\norm{\uu_{\mathrm{ls}}^\ell\cdot\nabla \uu_{\mathrm{ls}}^\ell}_{B^{d/2-1}} +\norm{Q(\rho^\ell,g^\ell)}_{B^{d/2-1}}\Big)d\tau \\ &
  \times\exp\left(C\int_0^t\Big(\norm{\uu^\ell}_{B^{d/2+1}}+ \norm{\uu_{\mathrm{ls}}^\ell}_{B^{d/2+1}} +(b_*\underline{\nu})^{1-2/\alpha}\bar{\nu}^{2/\alpha} \norm{\rho^\ell}_{B^{d/2+\alpha}}^{2/\alpha}\Big)d\tau\right),
\end{aligned}
\end{align}
and similarly
\begin{align}\label{eq.ns.10}
\begin{aligned}
  &\norm{\tilde{\uu}^\ell}_{\LL([0,t];B^{d/2-1+\alpha})}+\kappa b_*\underline{\nu} \norm{\tilde{\uu}^\ell}_{L^1([0,t];B^{d/2+1+\alpha})}\\
  \ls &\int_0^t\Big(\norm{\rho^\ell(\mu\Delta\uu_{\mathrm{ls}}^\ell +(\mu+\lambda)\nabla\dive\uu_{\mathrm{ls}}^\ell)}_{B^{d/2-1+\alpha}} \\ &\qquad\qquad+\norm{\uu_{\mathrm{ls}}^\ell\cdot\nabla \uu_{\mathrm{ls}}^\ell}_{B^{d/2-1+\alpha}} +\norm{Q(\rho^\ell,g^\ell)}_{B^{d/2-1+\alpha}}\Big)d\tau \\ &
  \times\exp\left(C\int_0^t\Big(\norm{\uu^\ell}_{B^{d/2+1}}+ \norm{\uu_{\mathrm{ls}}^\ell}_{B^{d/2+1}} +(b_*\underline{\nu})^{1-2/\alpha}\bar{\nu}^{2/\alpha} \norm{\rho^\ell}_{B^{d/2+\alpha}}^{2/\alpha}\Big)d\tau\right).
  \end{aligned}
\end{align}

By Lemma \ref{lem.fgbesov}, we get, for all $\sigma\in\{0,\alpha\}$, that
\begin{align*}
  &\norm{\rho^\ell\Delta\uu_{\mathrm{ls}}^\ell}_{B^{d/2-1+\sigma}}\lesssim \norm{\rho^\ell}_{B^{d/2}}\norm{\Delta\uu_{\mathrm{ls}}^\ell}_{B^{d/2-1+\sigma}}
  \lesssim \norm{\rho^\ell}_{B^{d/2}}\norm{\uu_{\mathrm{ls}}^\ell}_{B^{d/2+1+\sigma}},\\
  &\norm{\rho^\ell\nabla\dive\uu_{\mathrm{ls}}^\ell}_{B^{d/2-1+\sigma}}\lesssim \norm{\rho^\ell}_{B^{d/2}}\norm{\uu_{\mathrm{ls}}^\ell}_{B^{d/2+1+\sigma}},\\
  &\norm{\uu_{\mathrm{ls}}^\ell\cdot\nabla\uu_{\mathrm{ls}}^\ell}_{B^{d/2-1+\sigma}}
  \lesssim \norm{\uu_{\mathrm{ls}}^\ell}_{B^{d/2}} \norm{\nabla\uu_{\mathrm{ls}}^\ell}_{B^{d/2-1+\sigma}} \lesssim \norm{\uu_{\mathrm{ls}}^\ell}_{B^{d/2}} \norm{\uu_{\mathrm{ls}}^\ell}_{B^{d/2+\sigma}}.
\end{align*}
Recall that
\begin{align*}
  Q(\rho^\ell,g^\ell)
  =&\frac{\rho^\ell\nabla \rho^\ell}{\rho^\ell+1}-\nabla\rho^\ell+\frac{a_0}{\mb}(\rho^\ell+1)\nabla g^\ell+B(\rho^\ell,g^\ell)\Big[-\frac{\mb}{(\rho^\ell+1)^2}\nabla\rho^\ell\\
  &+\frac{k_0-a_0\nb}{\rho^\ell+1}\nabla\rho^\ell+\frac{a_0g^\ell\nabla\rho^\ell}{\rho^\ell+1} +\frac{a_0^2}{\mb}(g^\ell+g^\ell\rho^\ell)\nabla g^\ell+\frac{a_0(k_0+\mb+a_0\nb)}{\mb}\nabla g^\ell\\
  &+\frac{a_0k_0+a_0^2\nb}{\mb}\rho^\ell\nabla g^\ell\Big],
\end{align*}
where
\begin{align*}
  B(\rho^\ell,g^\ell):=\Big[\Big(\frac{\mb}{\rho^\ell+1}+a_0(g^\ell+\nb) -k_0\Big)^2+4k_0a_0(g^\ell+\nb)\Big]^{-1/2}.
\end{align*}
Similarly, for the third term of $Q$
\begin{align*}
  \norm{(\rho^\ell+1)\nabla g^\ell}_{B^{d/2-1+\sigma}}\lesssim (\norm{\rho^\ell}_{B^{d/2}}+1)\norm{\nabla g^\ell}_{B^{d/2-1+\sigma}}
  \lesssim (\norm{\rho^\ell}_{B^{d/2}}+1)\norm{g^\ell}_{B^{d/2+\sigma}}.
\end{align*}

By Lemmas \ref{lem.fgbesov} and \ref{lem.comp}, we get for the first two terms of $Q$
\begin{align*}
  \lnorm{\frac{\rho^\ell\nabla \rho^\ell}{\rho^\ell+1}-\nabla\rho^\ell}_{B^{d/2-1+\sigma}} \lesssim& \left(\lnorm{\frac{\rho^\ell}{\rho^\ell+1}}_{B^{d/2}}+1\right)\norm{\nabla \rho^\ell}_{B^{d/2-1+\sigma}}\\
  \lesssim& (\norm{\rho^\ell}_{B^{d/2}}+1)\norm{\rho^\ell}_{B^{d/2+\sigma}},
\end{align*}
and
\begin{align*}
  &\lnorm{B(\rho^\ell,g^\ell)\frac{\nabla\rho^\ell}{(\rho^\ell+1)^2}}_{B^{d/2-1+\sigma}}\\
  \lesssim& \lnorm{[B(\rho^\ell,g^\ell)-B(0,0)]\frac{\nabla\rho^\ell}{ (\rho^\ell+1)^2}}_{B^{d/2-1+\sigma}}+B(0,0)\lnorm{\frac{\nabla\rho^\ell}{ (\rho^\ell+1)^2}}_{B^{d/2-1+\sigma}}\\
  \lesssim&\left(\norm{B(\rho^\ell,g^\ell)-B(0,0)}_{B^{d/2}}+1\right) \left(\lnorm{\frac{\rho^\ell\rho^\ell+2\rho^\ell}{(\rho^\ell+1)^2}}_{B^{d/2}}+1\right) \norm{\nabla\rho^\ell}_{B^{d/2-1+\sigma}}\\
  \lesssim& \left(\norm{\rho^\ell}_{B^{d/2}}+\norm{g^\ell}_{B^{d/2}}+1\right)^2 \norm{\rho^\ell}_{B^{d/2+\sigma}}.
\end{align*}
Similarly,
\begin{align*}
  &\lnorm{B(\rho^\ell,g^\ell)\frac{\nabla \rho^\ell}{\rho^\ell+1}}_{B^{d/2-1+\sigma}}
  \lesssim \left(\norm{\rho^\ell}_{B^{d/2}}+\norm{g^\ell}_{B^{d/2}}+1\right)^2 \norm{\rho^\ell}_{B^{d/2+\sigma}},\\
  &\lnorm{B(\rho^\ell,g^\ell)\frac{g^\ell\nabla \rho^\ell}{\rho^\ell+1}}_{B^{d/2-1+\sigma}}
  \lesssim \left(\norm{\rho^\ell}_{B^{d/2}}+\norm{g^\ell}_{B^{d/2}}+1\right)^2 \norm{g^\ell}_{B^{d/2}}\norm{\rho^\ell}_{B^{d/2+\sigma}},\\
  &\norm{B(\rho^\ell,g^\ell)(1+\rho^\ell)g^\ell\nabla g^\ell}_{B^{d/2-1+\sigma}}
  \lesssim \left(\norm{\rho^\ell}_{B^{d/2}}+\norm{g^\ell}_{B^{d/2}}+1\right)^2 \norm{g^\ell}_{B^{d/2}}\norm{g^\ell}_{B^{d/2+\sigma}},\\
  &\norm{B(\rho^\ell,g^\ell)\nabla g^\ell}_{B^{d/2-1+\sigma}}
  \lesssim \left(\norm{\rho^\ell}_{B^{d/2}}+\norm{g^\ell}_{B^{d/2}}+1\right) \norm{g^\ell}_{B^{d/2+\sigma}},\\
  &\norm{B(\rho^\ell,g^\ell)\rho^\ell\nabla g^\ell}_{B^{d/2-1+\sigma}}
  \lesssim \left(\norm{\rho^\ell}_{B^{d/2}}+\norm{g^\ell}_{B^{d/2}}+1\right)
  \norm{\rho^\ell}_{B^{d/2}}\norm{g^\ell}_{B^{d/2+\sigma}}.
\end{align*}
Thus, we get
\begin{align*}
  \norm{Q(\rho^\ell,g^\ell)}_{B^{d/2-1+\sigma}}
  \lesssim \left(\norm{\rho^\ell}_{B^{d/2}}+\norm{g^\ell}_{B^{d/2}}+1\right)^3
  \left(\norm{\rho^\ell}_{B^{d/2+\sigma}}+\norm{g^\ell}_{B^{d/2+\sigma}}\right).
\end{align*}

Therefore, from \eqref{eq.mass.5}-\eqref{eq.ns.10}, we conclude that
\begin{align*}
  M^\ell(T)\lesssim & e^{C(U_{\mathrm{ls}}^\ell(T)+\tilde{U}^\ell(T)/(b_*\underline{\nu}))}(M_0+\frac{1}{3})\\ &\qquad+e^{C(U_{\mathrm{ls}}^\ell(T) +\tilde{U}^\ell(T)/(b_*\underline{\nu}))}(M_0+\frac{2}{3}) (U_{\mathrm{ls}}^\ell(T)+\tilde{U}^\ell(T)/(b_*\underline{\nu}))-\frac{1}{3},\\
  N^\ell(T)\lesssim & e^{C(U_{\mathrm{ls}}^\ell(T)+\tilde{U}^\ell(T)/(b_*\underline{\nu}))}(N_0+\frac{\nb}{3})\\ &\qquad+e^{C(U_{\mathrm{ls}}^\ell(T)+\tilde{U}^\ell(T)/(b_*\underline{\nu}))}(N_0+\frac{2\nb}{3}) (U_{\mathrm{ls}}^\ell(T)+\tilde{U}^\ell(T)/(b_*\underline{\nu}))-\frac{\nb}{3},\\
  \tilde{U}^\ell(T)\lesssim &\left((\bar{\nu}M^\ell(T)+U_0)U_{\mathrm{ls}}^\ell(T)+(M^\ell(T)+N^\ell(T)+1)^3(M^\ell(T) +N^\ell(T))T\right)\\
  &\qquad\times  e^{C[U_{\mathrm{ls}}^\ell(T)+\tilde{U}^\ell(T)/(b_*\underline{\nu}) +(b_*\underline{\nu})^{1-2/\alpha}\bar{\nu}^{2/\alpha}(M^\ell(T))^{2/\alpha}T]}.
\end{align*}
Now, if we take $T$ so small that
\begin{align*}
  \exp\left(CU_{\mathrm{ls}}^\ell(T)\right)\ls \sqrt{2}, \quad \exp\left(\frac{C\tilde{U}^\ell(T)}{b_*\underline{\nu}}\right)\ls\sqrt{2},
  \end{align*}
 and
 \begin{align*}
  \exp\left(C(b_*\underline{\nu})^{1-2/\alpha}\bar{\nu}^{2/\alpha} (M^\ell(T))^{2/\alpha}T\right)\ls 2,
\end{align*}
then we have
\begin{align}\label{eq.uni.1}
\left\{\begin{aligned}
  &M^\ell(T)\ls 4M_0+\frac{5}{3}, \quad N^\ell(T)\ls 4N_0+\frac{5\nb}{3}, \\
  &\tilde{U}^\ell(T)\ls C\left((M_0+N_0+1)^4(T+\bar{\nu}U_{\mathrm{ls}}^\ell(T))+U_0 U_{\mathrm{ls}}^\ell(T)\right).
\end{aligned}\right.
\end{align}

Noticing that $(\J \rho^\ell,\J g^\ell,\J\tilde{\uu}^\ell)=(\rho^\ell, g^\ell,\tilde{\uu}^\ell)$ by the construction of the approximated system. Thus, we have
\begin{align*}
  \partial_t(1+\rho^\ell)^{\pm 1}+\J(\uu^\ell\cdot\nabla(1+\rho^\ell)^{\pm 1})\pm\J((1+\rho^\ell)^{\pm } \dive \uu^\ell)=0.
\end{align*}
It follows, by noticing that $\abs{\partial_t\abs{f}}=\abs{\partial_t f}$, that
\begin{align*}
  \norm{(1+\rho^\ell)^{\pm 1}(t)}_\infty \ls&\norm{(1+\rho_0^\ell)^{\pm 1}}_\infty\\
  &+\int_0^t [\norm{\uu^\ell\cdot\nabla(1+\rho^\ell)^{\pm 1}}_\infty+\norm{(1+\rho^\ell)^{\pm } \dive \uu^\ell}_\infty] d\tau,
\end{align*}
which yields, by the Gronwall inequality, that
\begin{align*}
  &\norm{(1+\rho^\ell)^{\pm 1}(t)}_\infty
  \ls e^{\int_0^t\norm{\dive\uu^\ell}_\infty d\tau}\left(\norm{(1+\rho_0^\ell)^{\pm 1}}_\infty+\int_0^t\norm{\uu^\ell}_\infty\norm{\nabla \rho^\ell}_\infty d\tau\right)\\
  \ls&e^{\int_0^t\norm{\dive\uu^\ell}_\infty d\tau}\left(\norm{(1+\rho_0^\ell)^{\pm 1}}_\infty+C\int_0^t(\norm{\uu_{\mathrm{ls}}^\ell}_{B^{d/2}}+\norm{\tilde{\uu}^\ell}_{B^{d/2}}) \norm{\rho^\ell}_{B^{d/2+1}}d\tau\right)\\
  \ls&e^{\int_0^t\norm{\dive\uu^\ell}_\infty d\tau}\left(\norm{(1+\rho_0^\ell)^{\pm 1}}_\infty+CT(U_0+\tilde{U}^\ell(T))M^\ell(T)\right),
\end{align*}
where we have to choose $\alpha=1$ in the previous estimates. Hence, if we assume that there exist two positive constants $b_*$ and $b^*$ such that
\begin{align*}
  b_*\ls 1+\rho_0\ls b^*,
\end{align*}
then we can take $T$ small enough such that
\begin{align*}
  \int_0^T\norm{\dive\uu^\ell}_\infty d\tau\ls \ln 2, \text{ and } CT(U_0+\tilde{U}^\ell(T))M^\ell(T)\ls 1,
\end{align*}
and so
\begin{align}\label{eq.uni.2}
  \frac{b_*}{2(1+b_*)}\ls 1+\rho^\ell\ls 2(b^*+1).
\end{align}
Now, by means of a bootstrap argument, we can get that there exist two constants $\eta$ and $C$ depending only on $d$ such that if
\begin{align}\label{eq.uni.3}
  \left\{\begin{aligned}
    &(b_*\underline{\nu})^{1-2/\alpha}\bar{\nu}^{2/\alpha} (M^\ell(T))^{2/\alpha}T\ls \eta,\\
    &(M_0+N_0+1)^4(T+\bar{\nu}U_{\mathrm{ls}}^\ell(T))+U_0 U_{\mathrm{ls}}^\ell(T)\ls \eta b_*\underline{\nu},
  \end{aligned}\right.
\end{align}
then we have \eqref{eq.uni.1} and \eqref{eq.uni.2}.

Therefore, $T_\ell^*$ may be bounded from below by any time $T$ satisfying \eqref{eq.uni.3}, and the inequalities \eqref{eq.uni.1} and \eqref{eq.uni.2} are satisfied by $(\rho^\ell,g^\ell,\uu^\ell)$. In particular, $(\rho^\ell,g^\ell,\uu^\ell)_{\ell\in\N}$ is bounded in $F_T^1$.

\emph{Step 3: Time derivatives.} In order to pass to the limit in
the approximated system, we first give the following lemma.
\begin{lemma}\label{lem.timelocal}
  Let $\tilde{\rho}^\ell:=\rho^\ell-\J\rho_0$, $\tilde{g}^\ell:=g^\ell-\J g_0$. Then the sequences $(\tilde{\rho}^\ell)_{\ell\in\N}$ and $(\tilde{g}^\ell)_{\ell\in\N}$ are uniformly bounded in
  $$\mathcal{C}([0,T]; B^{d/2,d/2+1})\cap \mathcal{C}^{1/2}([0,T];B^{d/2-1,d/2}),$$
  and the sequence $(\tilde{\uu}^\ell)_{\ell\in\N}$ is uniformly bounded in
  $$(\mathcal{C}([0,T]; B^{d/2-1,d/2})\cap \mathcal{C}^{1/4}([0,T];B^{d/2-1,d/2}+B^{d/2-3/2,d/2-1/2}))^d.$$
\end{lemma}

\begin{proof}
From $\partial_t\tilde{\rho}^\ell=-\J(\uu^\ell\cdot\nabla \rho^\ell)+\J((\rho^\ell+1)\dive\uu^\ell)$, we have
\begin{align*}
    &\norm{\partial_t\tilde{\rho}^\ell}_{L^2([0,T];B^{d/2-1,d/2})} \ls\norm{\partial_t\tilde{\rho}^\ell}_{\LL[2]([0,T];B^{d/2-1,d/2})}\\
    \ls& \norm{\uu^\ell\cdot\nabla \rho^\ell}_{\LL[2]([0,T];B^{d/2-1,d/2})} +\norm{(\rho^\ell+1)\dive\uu^\ell}_{\LL[2]([0,T];B^{d/2-1,d/2})}\\
    \lesssim &\norm{\uu^\ell}_{\LL[2]([0,T];B^{d/2})}\norm{\rho^\ell}_{\LL([0,T];B^{d/2,d/2+1})} \\ &+(\norm{\rho^\ell}_{\LL([0,T];B^{d/2})}+1)\norm{\uu^\ell}_{\LL[2]([0,T];B^{d/2,d/2+1})}.
\end{align*}
Since $(\uu^\ell)_{\ell\in\N}$ is uniformly bounded in
$$\LL([0,T];B^{d/2-1,d/2})\cap L^1([0,T];B^{d/2+1,d/2+2}),$$
it is also bounded in $\LL[2]([0,T];B^{d/2,d/2+1})$ by Lemma \ref{lem.inter}. Recall that $(\rho^\ell)_{\ell\in\N}$ is uniformly bounded in $\LL([0,T];B^{d/2,d/2+1})$, then $(\partial_t\tilde{\rho}^\ell)_{\ell\in\N}$ is uniformly bounded in $$L^2([0,T];B^{d/2-1,d/2}),$$
and so $(\tilde{\rho}^\ell)_{\ell\in\N}$ is uniformly bounded in $$\mathcal{C}^{1/2}([0,T];B^{d/2-1,d/2}) \text{ and in } \mathcal{C}([0,T]; B^{d/2,d/2-1}).$$ Similarly, we have the same arguments for $\tilde{g}^\ell$.

Recall that
\begin{align*}
\partial_t\tilde{\uu}^\ell=&-\J(\uu^\ell\cdot\nabla \uu^\ell)+\J[(1+\rho^\ell)(\mu\Delta\tilde{\uu}^\ell +(\mu+\lambda)\nabla \dive\tilde{\uu}^\ell)]\\
    &\qquad+\J[\rho^\ell(\mu\Delta\uu_{\mathrm{ls}}^\ell +(\mu+\lambda)\nabla\dive\uu_{\mathrm{ls}}^\ell)]-\J Q(\rho^\ell,g^\ell).
\end{align*}
Since
\begin{align*}
  &\norm{\uu^\ell\cdot\nabla\uu^\ell}_{L^2([0,T];B^{d/2-1,d/2})}
  \lesssim \norm{\uu^\ell}_{\LL([0,T];B^{d/2-1,d/2})}\norm{\uu^\ell}_{L^2([0,T];B^{d/2+1})},\\
  &\norm{(1+\rho^\ell)(\mu\Delta\tilde{\uu}^\ell +(\mu+\lambda)\nabla \dive\tilde{\uu}^\ell)}_{L^{4/3}([0,T];B^{d/2-3/2,d/2-1/2})}\\
  \lesssim& (1+\norm{\rho^\ell}_{\LL([0,T];B^{d/2})} ) \norm{\tilde{\uu}^\ell}_{L^{4/3}([0,T];B^{d/2+1/2,d/2+3/2})},\\
  &\norm{\rho^\ell(\mu\Delta\uu_{\mathrm{ls}}^\ell +(\mu+\lambda)\nabla\dive\uu_{\mathrm{ls}}^\ell)}_{L^{4/3}([0,T];B^{d/2-3/2,d/2-1/2})} \\ \lesssim& \norm{\rho^\ell}_{\LL([0,T];B^{d/2})} \norm{\uu_{\mathrm{ls}}^\ell}_{L^{4/3}([0,T];B^{d/2+1/2,d/2+3/2})},\\
  &\norm{Q(\rho^\ell,g^\ell)}_{\LL([0,T];B^{d/2-1,d/2})}\\
  \lesssim & \left(\norm{\rho^\ell}_{\LL([0,T];B^{d/2}}+\norm{g^\ell}_{\LL([0,T];B^{d/2})}+1\right)^3
  \left(\norm{\rho^\ell}_{\LL([0,T];B^{d/2,d/2+1})} \right.\\ &\left.\qquad+\norm{g^\ell}_{\LL([0,T];B^{d/2,d/2+1})}\right),
\end{align*}
by Lemma \ref{lem.inter}, we can conclude that $(\tilde{\uu}^\ell)_{\ell\in\N}$ is uniformly bounded in
$$\mathcal{C}^{1/4}([0,T];B^{d/2-1,d/2}+B^{d/2-3/2,d/2-1/2}) \text{ and in } \mathcal{C}([0,T]; B^{d/2-1,d/2}).$$
This completes the proof of the lemma.
\end{proof}

\emph{Step 4: Compactness and convergence.} The proof is based on the
Arzel\`a-Ascoli theorem and compact embeddings for Besov spaces.
Since it is similar to the arguments for global well-posedness, we
only give the outlines of the proof.

From Lemma \ref{lem.timelocal}, $(\tilde{\rho}^\ell)_{\ell\in\N}$ is uniformly bounded in the space
$$\LL([0,T];B^{d/2,d/2+1})$$
and equicontinuous on $[0,T]$ with values in $B^{d/2-1,d/2}$. Since the embedding $$B^{d/2-1,d/2}\hookrightarrow B^{d/2-1}$$
is (locally) compact, and $(\rho_0^\ell)_{\ell\in\N}$ tends to $\rho_0$ in $B^{d/2,d/2+1}$, we conclude that $(\rho^\ell)_{\ell\in\N}$ tends (up to an extraction) to some distribution $\rho$. Given that $(\rho^\ell)_{\ell\in\N}$ is uniformly bounded in $\LL([0,T];B^{d/2})$, we actually have $$\rho\in \LL([0,T];B^{d/2}).$$
 The same arguments are valid for the sequence $(g^\ell)_{\ell\in\N}$.

From the definition of $(\uu_{\mathrm{ls}}^\ell)_{\ell\in\N}$, it is clear that $(\uu_{\mathrm{ls}}^\ell)_{\ell\in\N}$ tends to the solution $\uu_{\mathrm{ls}}$ of \eqref{eq.ls} in $\LL([0,t];B^{d/2-1,d/2})\cap L^1([0,T]; B^{d/2+1,d/2+2})$.

Since $(\tilde{\uu}^\ell)_{\ell\in\N}$ is uniformly bounded in $\LL([0,T]; B^{d/2-1,d/2})$ and equicontinuous on $[0,T]$ with values in $B^{d/2-1,d/2}+B^{d/2-3/2,d/2-1/2}$, it enable us to conclude that $(\tilde{\uu}^\ell)_{\ell\in\N}$ converges, up to an extraction, to some function $\tilde{\uu}\in \LL([0,T];B^{d/2-1})\cap L^1([0,T];B^{d/2+1})$.

Thus, we can pass to the limit in the system \eqref{eq.system.app4} and setting $\uu:=\tilde{\uu}+\uu_{\mathrm{ls}}$. Then, $(\rho,g,\uu)$ satisfies the system \eqref{eq.system.2}.

\emph{Step 5: Continuities in time.}

From the first equation of \eqref{eq.system.2}, we get $\partial_t\rho\in L^2([0,T];B^{d/2-1,d/2})$ which implies $\rho\in \mathcal{C}([0,T];B^{d/2-1,d/2})$. So does $g$ in the same space. For $\uu$, we can derive, from the third equation of \eqref{eq.system.2}, that $\partial_t\uu\in (L^1+L^2)([0,T]; B^{d/2-1,d/2})$ which yields $\uu\in \mathcal{C}([0,T];B^{d/2-1,d/2})$.

\subsection{Uniqueness}

Let $(\rho_1,g_1,\uu_1)$ and $(\rho_2,g_2,\uu_2)$ be two solutions in $F_T^1$ of \eqref{eq.system.2} with the same initial data. Without loss of generality, we can assume that $(\rho_2,g_2,\uu_2)$ is the solution constructed in the previous subsection such that
$$1+\inf_{(t,x)\in[0,T]\times\R^d} \rho_2(t,x)>0.$$
We want to prove that $(\rho_2,g_2,\uu_2)\equiv (\rho_1,g_1,\uu_1)$ on $[0,T]\times\R^d$. To this goal, we shall estimate the discrepancy $(\delta\rho,\delta g,\delta\uu):=(\rho_2-\rho_1,g_2-g_1,\uu_2-\uu_1)$ with respect to a suitable norm, satisfying
\begin{align}\label{eq.sys.uni}
  \left\{\begin{aligned}
    &\partial_t\delta\rho+\uu_2\cdot\nabla\delta\rho+\delta\uu\cdot\nabla\rho_1 =\delta\rho\dive\uu_2+(\rho_1+1)\dive\delta\uu,\\
    &\partial_t\delta g+\uu_2\cdot\nabla\delta g+\delta\uu\cdot\nabla g_1 =-\delta g\dive\uu_2-(g_1+\nb)\dive\delta\uu,\\
    &\partial_t\delta\uu+\uu_2\cdot\nabla\delta\uu+\delta\uu\cdot\nabla\uu_1 -(1+\rho_2)(\mu\Delta \delta\uu_(\mu+\lambda)\nabla\dive\uu)\\
    &\qquad-\delta\rho(\mu\Delta\uu_1+(\mu+\lambda)\nabla\dive\uu_1)
    +Q(\rho_2,g_2)-Q(\rho_1,g_1)=0,\\
    &(\delta\rho,\delta g,\delta\uu)|_{t=0}=(0,0,\mathbf{0}).
  \end{aligned}\right.
\end{align}
We shall prove the uniqueness in a larger function space
\begin{align*}
  F_T:=(\mathcal{C}([0,T];B^{d/2}))^{1+1}\times(\mathcal{C}([0,T];B^{d/2})\cap L^1([0,T]; B^{d/2+1}))^d.
\end{align*}

By Proposition \ref{prop.3}, we get for all $T'\in[0,T]$
\begin{align*}
  &\norm{\delta\rho}_{\LL([0,T'];B^{d/2})}\\
  \ls& e^{C\norm{\uu_2}_{L^1([0,T'];B^{d/2+1})}}\int_0^{T'}\Big(\norm{\delta\uu\cdot\nabla \rho_1}_{B^{d/2}}+\norm{\delta\rho\dive\uu_2}_{B^{d/2}} \\ &\qquad\qquad\qquad\qquad\qquad\qquad\qquad +\norm{(\rho_1+1)\dive\delta\uu)}_{B^{d/2}}\Big)d\tau\\
  \lesssim& e^{C\norm{\uu_2}_{L^1([0,T'];B^{d/2+1})}}\int_0^{T'} \Big[\norm{\delta\uu}_{B^{d/2}} \norm{\rho_1}_{B^{d/2+1}}+\norm{\delta\rho}_{B^{d/2}}\norm{\uu_2}_{B^{d/2+1}} \\ &\qquad\qquad\qquad\qquad\qquad +(1+\norm{\rho_1}_{B^{d/2}})\norm{\delta\uu}_{B^{d/2+1}}\Big]d\tau.
\end{align*}
Using the Gronwall inequality, it yields
\begin{align}\label{eq.uni.4}
\begin{aligned}
  &\norm{\delta\rho}_{\LL([0,T'];B^{d/2})}\\
  \lesssim &e^{C\norm{\uu_2}_{L^1([0,T'];B^{d/2+1})}}\int_0^{T'}\big[\norm{\delta\uu}_{B^{d/2}} \norm{\rho_1}_{B^{d/2+1}}+(1+\norm{\rho_1}_{B^{d/2}}) \norm{\delta\uu}_{B^{d/2+1}}\big]d\tau\\
  \ls &C_T\left(\norm{\delta\uu}_{L^2([0,T']; B^{d/2})}+\norm{\delta\uu}_{L^1([0,T'];B^{d/2+1})}\right),
\end{aligned}
\end{align}
where $C_T$ is independent of $T'$.

Similarly, we have
\begin{align}\label{eq.uni.5}
\norm{\delta g}_{\LL([0,T'];B^{d/2})}\ls C_T\left(\norm{\delta\uu}_{L^2([0,T']; B^{d/2})}+\norm{\delta\uu}_{L^1([0,T'];B^{d/2+1})}\right).
\end{align}

Applying Lemmas \ref{lem.moment} and \ref{lem.fgbesov} to the third equation of \eqref{eq.sys.uni}, it yields
\begin{align*}
  &\norm{\delta\uu}_{\LL([0,T'];B^{d/2-1})}+\norm{\delta\uu}_{L^1([0,T'];B^{d/2+1})}\\
  \ls& C e^{C\int_0^{T'}[\norm{\uu_1}_{B^{d/2+1}}+\norm{\uu_2}_{B^{d/2+1}}]d\tau}\int_0^{T'} \Big(\norm{\delta\rho}_{B^{d/2}}\norm{\uu_1}_{B^{d/2+1}} \\ &\qquad\qquad\qquad\quad\qquad\qquad\qquad\qquad\qquad+\norm{Q(\rho_2,g_2)-Q(\rho_1,g_1)}_{B^{d/2-1}}\Big)d\tau.
\end{align*}
By Lemma \ref{lem.comp}, we get
\begin{align*}
  \norm{Q(\rho_2,g_2)-Q(\rho_1,g_1)}_{B^{d/2-1}}\lesssim (1+\norm{(\rho_1,\rho_2,g_1,g_2)}_{B^{d/2}})^3(\norm{\delta\rho}_{B^{d/2}}+\norm{\delta g}_{B^{d/2}}).
\end{align*}
Thus, it follows that
\begin{align}\label{eq.uni.6}
\begin{aligned}
 &\norm{\delta\uu}_{\LL([0,T'];B^{d/2-1})}+\norm{\delta\uu}_{L^1([0,T'];B^{d/2+1})}\\
  \ls& C_T(T'+T'^{1/2})(\norm{\delta\rho}_{L^\infty([0,T'];B^{d/2})}+\norm{\delta g}_{L^\infty([0,T'];B^{d/2})}),
\end{aligned}
\end{align}
since $\uu_1\in L^2([0,T];B^{d/2+1})$ by Lemma \ref{lem.inter}.

From \eqref{eq.uni.4}-\eqref{eq.uni.6}, it yields, with the help of Lemma \ref{lem.inter}, that
\begin{align*}
  &\norm{\delta\uu}_{\LL([0,T'];B^{d/2-1})}+\norm{\delta\uu}_{L^1([0,T'];B^{d/2+1})}\\
  \ls& C_T(T'+T'^{1/2})(\norm{\delta\uu}_{L^2([0,T']; B^{d/2})}+\norm{\delta\uu}_{L^1([0,T'];B^{d/2+1})})\\
  \ls& C_T(T'+T'^{1/2})(\norm{\delta\uu}_{\LL([0,T']; B^{d/2-1})}+\norm{\delta\uu}_{L^1([0,T'];B^{d/2+1})}).
\end{align*}
Therefore, if we choose $T'$ so small that $C_T(T'+T'^{1/2})<1$, then we obtain that $(\delta\rho,\delta g,\delta\uu)=(0,0,\mathbf{0})$ on the time interval $[0,T']$. As in the proof of uniqueness for global well-posedness, we can extend $T'$ to $T$ by the translation with respect to the time variable, i.e. $(\delta\rho,\delta g,\delta\uu)=(0,0,\mathbf{0})$ on the time interval $[0,T]$.

\subsection{A continuation criterion}

\begin{proposition}\label{prop.cc}
  Under the hypotheses of Theorem \ref{thm.4}, assume that the system \eqref{eq.system.2} has a solution $(\rho,g,\uu)$ on $[0,T)\times\R^d$ which belongs to $F_{T'}^1$ for all $T'<T$ and satisfies
  \begin{align*}
    \rho, g\in L^\infty([0,T);B^{d/2,d/2+1}), \quad \inf_{(t,x)\in [0,T)\times\R^d} \rho(t,x)>-1, \quad \int_0^T \norm{\nabla\uu}_{\infty}dt<\infty.
  \end{align*}
  Then, there exists some $T^*>T$ such that $(\rho,g,\uu)$ may be continued on $[0,T^*]\times\R^d$ to a solution of \eqref{eq.system.2} which belongs to $F_{T^*}^1$.
\end{proposition}

\begin{proof}
  Recall that $\uu$ satisfies
  \begin{align*}
    \uu_t+\uu\cdot\nabla\uu-(1+\rho)(\mu\Delta\uu+(\mu+\lambda)\nabla\dive\uu)+Q(\rho,g)=0, \quad \uu|_{t=0}=\uu_0.
  \end{align*}
  By Lemma \ref{lem.moment}, we get, for $T'<T$, that
  \begin{align*}
    &\norm{\uu}_{\LL([0,T'];B^{d/2-1,d/2})} +\underline{\nu}\norm{\uu}_{L^1([0,T'];B^{d/2+1,d/2+2})}\\
    \ls&C e^{C\int_0^{T'}\left(\norm{\nabla\uu}_\infty +\norm{\rho}_{B^{d/2+1}}^2\right)dt}\left(\norm{\uu_0}_{B^{d/2-1,d/2}}+\int_0^{T'} \norm{\rho}_{B^{d/2,d/2+1}}dt\right)
  \end{align*}
  for some constant $C$ depending only on $d$ and viscosity coefficients. Thus, there exists a constant $\eps>0$ such that \eqref{eq.system.2} with initial data $(\rho(T-\eps),g(T-\eps),\uu(T-\eps))$ yields a solution on $[0,2\eps]$. Since the solution $(\rho,g,\uu)$ is unique on $[0,T)$, this provides a continuation of $(\rho,g,\uu)$ beyond $T$.
\end{proof}

\begin{appendix}

\section{Littlewood-Paley theory and Besov spaces}
\renewcommand{\thesection}{A}

This section is devoted to recall some properties of Littlewood-Paley theory and Besov spaces which will be used in this paper. For more details, one can see \cite{Danchin01,HHL} and references therein.

Let $\psi : \R^d \to [0,1]$ be a radial smooth cut-off function
valued in $[0,1]$ such that
\begin{align*}
    \psi(\xi)=\left\{
    \begin{array}{ll}
    1, &\abs{\xi}\ls 3/4,\\
    \text{smooth}, &3/4<\abs{\xi}<4/3,\\
    0, &\abs{\xi}\gs 4/3.
    \end{array}
    \right.
\end{align*}
Let $\varphi(\xi)$ be the function
\begin{align*}
    \varphi(\xi):=\psi(\xi/2)-\psi(\xi).
\end{align*}
Thus, $\psi$ is supported in the ball $\set{\xi\in\R^d:
\abs{\xi}\ls 4/3}$, and $\varphi$ is also a smooth cut-off function
valued in $[0,1]$ and supported in the annulus $\{\xi\in\R^d:
3/4\ls\abs{\xi}\ls 8/3\}$.
 By construction, we have
\begin{align*}
    \sum_{k\in\Z}\varphi(2^{-k}\xi)=1, \quad \forall
    \xi\neq 0.
\end{align*}
One can define the dyadic blocks as follows. For $k\in\Z$, let
\begin{align*}
\dk f:=\F^{-1}\varphi(2^{-k}\xi)\F f.
\end{align*}
The formal decomposition
\begin{align}\label{lpd}
    f=\sum_{k\in\Z}\dk f
\end{align}
is called homogeneous Littlewood-Paley decomposition. Nevertheless, \eqref{lpd} is true modulo
polynomials, in other words (cf.\cite{Pee76}), if
$f\in\Sz'(\R^d)$, then $\sum_{k\in\Z}\dk f$ converges modulo
$\mathscr{P}[\R^d]$ and \eqref{lpd} holds in
$\Sz'(\R^d)/\mathscr{P}[\R^d]$.

\begin{definition}
Let $s\in\R$. For $f\in\Sz'(\R^d)$, we
write
\begin{align*}
    \norm{f}_{\be[s]}=\sum_{k\in\Z}
    2^{ks}\norm{\dk f}_{2}.
\end{align*}
\end{definition}

A difficulty comes from the choice of homogeneous spaces at this point.
Indeed, $\norm{\cdot}_{\be[s]}$ cannot be a norm on
$\{f\in\Sz'(\R^d): \norm{f}_{\be[s]}<\infty\}$ because
$\norm{f}_{\be[s]}=0$ means that $f$ is a polynomial. This
enforces us to adopt the following definition for homogeneous Besov
spaces (cf. \cite{Danchin01}).

\begin{definition}
Let $s\in\R$ and $m=-[\hn+1-s]$. If $m<0$, then we define
$\be[s](\R^d)$ as
\begin{align*}
    \be[s]=\Big\{f\in\Sz'(\R^d):
    \norm{f}_{\be[s]}<\infty \text{ and } f=\sum_{k\in\Z}\dk f
    \text{ in } \Sz'(\R^d)\Big\}.
\end{align*}
If $m\gs 0$, we denote by $\mathscr{P}_m$ the set of $d$ variables
polynomials of degree less than or equal to $m$ and define
\begin{align*}
    \be[s]=\Big\{f\in\Sz'(\R^d)/\mathscr{P}_m:
    \norm{f}_{\be[s]}<\infty \text{ and } f=\sum_{k\in\Z}\dk f
    \text{ in } \Sz'(\R^d)/\mathscr{P}_m\Big\}.
\end{align*}
\end{definition}

For the composition of functions, we have the following estimates.

\begin{lemma}\label{lem.comp}
Let $s>0$ and $u\in \be[s]\cap L^\infty$. Then, it holds

\mbox{\rm (i) } Let $F\in W_{loc}^{[s]+2,\infty}(\R^d)$ with
$F(0)=0$. Then $F(u)\in\be[s]$. Moreover, there exists a function of
one variable $C_0$ depending only on $s$ and $F$, and such that
\begin{align*}
    \norm{F(u)}_{\be[s]}\ls
    C_0(\norm{u}_{L^\infty})\norm{u}_{\be[s]}.
\end{align*}

\mbox{\rm (ii)} If $u,\, v\in\be[\hn]$, $(v-u)\in \be[s]$ for
$s\in(-\hn,\hn]$ and $G\in W_{loc}^{[\hn]+3,\infty}(\R^d)$
satisfies $G'(0)=0$, then $G(v)-G(u)\in \be[s]$ and there exists a
function of two variables $C$ depending only on $s$, $N$ and $G$,
and such that
\begin{align*}
    \norm{G(v)-G(u)}_{\be[s]}\ls C(\norm{u}_{L^\infty},
    \norm{v}_{L^\infty})\left(\norm{u}_{\be[\hn]}+\norm{v}_{\be[\hn]}\right)
    \norm{v-u}_{\be[s]}.
\end{align*}
\end{lemma}

We also need hybrid Besov spaces for which regularity assumptions
are different in low frequencies and high frequencies \cite{Danchin01}.
We are going to recall the definition of these new spaces and some
of their main properties.

\begin{definition}
Let $s,\,t\in\R$. We define
\begin{align*}
    \norm{f}_{\hbe[s]{t}}=\sum_{k\ls 0}2^{ks}\norm{\dk f}_{2}
    +\sum_{k>0}2^{kt}\norm{\dk f}_{2}.
\end{align*}
Let $m=-[\hn+1-s]$, we then define
\begin{align*}
    \hbe[s]{t}(\R^d)
 =&\set{f\in\Sz'(\R^d): \norm{f}_{\hbe[s]{t}}<\infty},
  \quad \text{if } m<0,\\
    \hbe[s]{t}(\R^d)
 =&\set{f\in\Sz'(\R^d)/\mathscr{P}_m: \norm{f}_{\hbe[s]{t}}<\infty},
   \quad \text{if } m \gs 0.
\end{align*}
\end{definition}

\begin{lemma} We have the following inclusions for hybrid Besov spaces.

\mbox{\rm (i)} We have $\hbe[s]{s}=\be[s]$.

\mbox{\rm (ii)} If $s\ls t$ then $\hbe[s]{t}=\be[s]\cap\be[t]$.
Otherwise, $\hbe[s]{t}=\be[s]+\be[t]$.

\mbox{\rm (iii)} The space $\hbe[0]{s}$ coincides with the usual
inhomogeneous Besov space $B_{2,1}^s$.

\mbox{\rm (iv)} If $s_1\ls s_2$ and $t_1\gs t_2$, then
$\hbe[s_1]{t_1}\hookrightarrow \hbe[s_2]{t_2}$.
\end{lemma}

Let us now recall some useful estimates for the product in hybrid
Besov spaces.

\begin{lemma}\label{lem.fgbesov}
Let $s_1,\, s_2>0$ and $f,\,g\in L^\infty\cap \hbe[s_1]{s_2}$. Then
$fg\in \hbe[s_1]{s_2}$ and
\begin{align*}
    \norm{fg}_{\hbe[s_1]{s_2}}\lesssim
    \norm{f}_{L^\infty}\norm{g}_{\hbe[s_1]{s_2}}
    +\norm{f}_{\hbe[s_1]{s_2}}\norm{g}_{L^\infty}.
\end{align*}
Let $s\in (-d/2,d/2]$, $f\in B^{d/2}$ and $g\in B^s$, then $fg\in B^s$ and
$$\norm{fg}_{B^s}\lesssim \norm{f}_{B^{d/2}}\norm{g}_{B^s}.$$

Let $s_1, s_2, t_1, t_2\ls \hn$ such that $\min(s_1+s_2, t_1+t_2)>0$,
$f\in \hbe[s_1]{t_1}$ and $g\in
\hbe[s_2]{t_2}$. Then $fg\in \hbe[s_1+s_2-1]{t_1+t_2-1}$ and
\begin{align*}
    \norm{fg}_{\hbe[s_1+s_2-\hn]{t_1+t_2-\hn}}\lesssim
    \norm{f}_{\hbe[s_1]{t_1}}\norm{g}_{\hbe[s_2]{t_2}}.
\end{align*}
\end{lemma}

For $\alpha,\beta\in\R$, let us define the following characteristic function on $\Z$:
\begin{align*}
     \tilde{\varphi}^{\alpha,\beta}(k)=\left\{
     \begin{array}{ll}
        \alpha,\quad &\text{if } k\ls 0,\\
        \beta, & \text{if } k\gs 1.
     \end{array}
     \right.
\end{align*}
Then, we can recall the following lemma.
\begin{lemma}\label{lem.innner}
    Let $F$ be an homogeneous smooth function of degree $m$. Suppose that $-d/2<s_1,t_1,s_2,t_2\ls 1+d/2$. The following two estimates hold:
    \begin{align*}
    &\abs{(F(D)\dk(\vv\cdot \nabla a),F(D)\dk a)}\\
    &\qquad\qquad\lesssim \gamma_k2^{-k(\tilde{\varphi}^{s_1,s_2}(k)-m)}\norm{\vv}_{\be[\hn+1]} \norm{a}_{\hbe[s_1]{s_2}}\norm{F(D)\dk a}_{2},\\
    &\abs{(F(D)\dk(\vv\cdot\nabla a),\dk b)+(\dk(\vv\cdot\nabla b),F(D)\dk a)}\\
    &\qquad\qquad\lesssim \gamma_k\norm{\vv}_{\be[\hn+1]}\times\big( 2^{-k\tilde{\varphi}^{t_1,t_2}(k)}\norm{F(D)\dk a}_{2}\norm{b}_{\hbe[t_1]{t_2}}\\
    &\qquad\qquad\qquad+2^{-k(\tilde{\varphi}^{s_1,s_2}(k)-m)}\norm{a}_{\hbe[s_1]{s_2}}
    \norm{\dk b}_{2}\big),
    \end{align*}
where $(\cdot,\cdot)$ denotes the $2$-inner product,  $\sum_{k\in\Z} \gamma_k\ls 1$ and the operator $F(D)$ is defined by $F(D)f:=\F^{-1} F(\xi)\F f$.
\end{lemma}

In the context of this paper, we also need to use the interpolation spaces of hybrid Besov spaces together with a time space such as $L^p([0,T);\hbe[s]{t})$. Thus, we have to introduce the Chemin-Lerner type space (cf. \cite{CheLer}) which is a refinement of the space $L^p([0,T);\hbe[s]{t})$.

\begin{definition}
Let $p\in [1,\infty]$, $T\in(0,\infty]$ and $s_1,\, s_2\in\R$. Then we define
\begin{align*}
    \norm{f}_{\LL[p]([0,T);\hbe[s]{t})}=\sum_{k\ls 0}2^{ks}\norm{\dk f}_{L^p([0,T); L^2)}+\sum_{k> 0}2^{kt}\norm{\dk f}_{L^p([0,T); L^2)}.
\end{align*}
\end{definition}

Noting that Minkowski's inequality yields $\norm{f}_{L^p([0,T);\hbe[s]{t})}\ls \norm{f}_{\LL[p]([0,T);\hbe[s]{t})}$, we define spaces $\LL[p]([0,T);\hbe[s]{t})$ as follows
\begin{align*}
    \LL[p]([0,T);\hbe[s]{t})=\set{f\in L^p([0,T);\hbe[s]{t}): \norm{f}_{\LL[p]([0,T);\hbe[s]{t})}<\infty}.
\end{align*}
If $T=\infty$, then we omit the subscript $T$ from the notation $\LL[p]([0,T);\hbe[s]{t})$, that is, $\LL[p](\hbe[s]{t})$ for simplicity. We will denote by $\tilde{\mathcal{C}}([0,T);\hbe[s]{t})$ the subset of functions of $\LL([0,T);\hbe[s]{t})$ which are continuous on $[0,T)$ with values in $\hbe[s]{t}$.

Let us observe that $L^1([0,T);\hbe[s]{t})=\LL[1]([0,T);\hbe[s]{t})$, but the embedding $$\LL[p]([0,T);\hbe[s]{t})\subset L^p([0,T);\hbe[s]{t})$$
is strict if $p>1$.

We will use the following interpolation property which can be verified easily (cf. \cite{BCDbook,Ber76}).

\begin{lemma}\label{lem.inter}
Let $s,t,s_1,t_1, s_2,t_2\in\R$ and $p,p_1,p_2\in[1,\infty]$. We have
\begin{align*}
    \norm{f}_{\LL[p]([0,T);\hbe[s]{t})}\ls \norm{f}_{\LL[{{p_1}}]([0,T);\hbe[s_1]{t_1})}^\theta \norm{f}_{\LL[{{p_2}}]([0,T);\hbe[s_2]{t_2})}^{1-\theta},
\end{align*}
where $\frac{1}{p}=\frac{\theta}{p_1}+\frac{1-\theta}{p_2}$, $s=\theta s_1+(1-\theta)s_2$ and $t=\theta t_1+(1-\theta)t_2$.
\end{lemma}
\end{appendix}


\section*{Acknowledgments}
The authors would like to thank the referees for their valuable comments and Dr. L. Yao for helpful comments on the original version of the manuscript. Hao's work was partially supported by the National Natural Science Foundation of China (grant 11171327), and the Youth Innovation Promotion Association, Chinese Academy of Sciences. Li's work was partially supported by NSFC grant 11171228 and 11011130029, and the AHRDIHL Project of Beijing Municipality (No. PHR 201006107).

\end{document}